\documentclass[a4paper,11pt]{article}
\usepackage{latexsym, amssymb, amsfonts,amsmath}
\usepackage{amssymb,latexsym}
\usepackage{amsfonts}
\usepackage{amsmath,amsthm,graphicx}
\newcommand{\be}{\begin{equation}}
\newcommand{\fonedown}{f_{_{1}}^}
\newcommand{\fzerodown}{f_{_{0}}^}
\newcommand{\wtil}{\widetilde}
\newcommand{\inner}{\innerr_{L^2(h)}}
\newcommand{\ee}{\end{equation}}
\newcommand{\bea}{\begin{eqnarray}}
\newcommand{\eea}{\end{eqnarray}}
\newcommand{\bean}{\begin{eqnarray*}}
\newcommand{\eean}{\end{eqnarray*}}
\newcommand{\brray}{\begin{array}}
\newcommand{\erray}{\end{array}}
\newcommand{\ben}{\begin{equation}{nonumber}}
\newcommand{\een}{\end{equation}{nonumber}}

\newtheorem{dfn}{Definition}[section]
\newtheorem{thm}[dfn]{Theorem}

\newtheorem{lmma}[dfn]{Lemma}

\newtheorem{ppsn}[dfn]{Proposition}
\newtheorem{crlre}[dfn]{Corollary}
\newtheorem{xmpl}[dfn]{Example}
\newtheorem{rmrk}[dfn]{Remark}
\newcommand{\bdfn}{\begin{dfn}}
\newcommand{\bthm}{\begin{thm}}

\newcommand{\ota}{\ot_{alg}}

\newcommand{\blr}{\begin{list}{$($\roman{cnt1}$)$} {\usecounter{cnt1}
        \setlength{\topsep}{0pt} \setlength{\itemsep}{0pt}}}
\newcommand{\bla}{\begin{list}{$($\alph{cnt2}$)$} {\usecounter{cnt2}
       \setlength{\topsep}{0pt} \setlength{\itemsep}{0pt}}}
\newcommand{\bln}{\begin{list}{$($\arabic{cnt3}$)$} {\usecounter{cnt3}
                \setlength{\topsep}{0pt} \setlength{\itemsep}{0pt}}}
\newcommand{\el}{\end{list}}

\newcommand{\blmma}{\begin{lmma}}
\newcommand{\bppsn}{\begin{ppsn}}
\newcommand{\bcrlre}{\begin{crlre}}
\newcommand{\bxmpl}{\begin{xmpl}}
\newcommand{\brmrk}{\begin{rmrk}}
\newcommand{\edfn}{\end{dfn}}
\newcommand{\ethm}{\end{thm}}
\newcommand{\elmma}{\end{lmma}}

\newcommand{\till}{\widetilde{\cll}}
\newcommand{\eppsn}{\end{ppsn}}
\newcommand{\ecrlre}{\end{crlre}}
\newcommand{\exmpl}{\end{xmpl}}
\newcommand{\ermrk}{\end{rmrk}}

\newcommand{\IC}{\mathbb{C}}

\newcommand{\IE}{{I\! \! E}}

\newcommand{\IN}{{I\! \! N}}

\newcommand{\IP}{{I\! \! P}}
\newcommand{\IQ}{\mathbb{Q}}

\newcommand{\IR}{\mathbb{R}}

\newcommand{\IT}{\mathbb{T}}

\newcommand{\IZ}{\mathbb{Z}}

\newcommand{\innerl}{\left\langle}
\newcommand{\innerr}{\right\rangle}

\newcommand{\cla}{{\cal A}}
\newcommand{\clb}{{\cal B}}
\newcommand{\clc}{{\cal C}}
\newcommand{\cld}{{\cal D}}

\newcommand{\clf}{{\cal F}}
\newcommand{\clg}{{\cal G}}
\newcommand{\clh}{{\cal H}}

\newcommand{\clj}{{\cal J}}
\newcommand{\clk}{{\cal K}}
\newcommand{\cll}{{\cal L}}
\newcommand{\clm}{{\cal M}}

\newcommand{\clp}{{\cal P}}
\newcommand{\clq}{{\cal Q}}

\newcommand{\clv}{{\cal V}}
\newcommand{\clw}{{\cal W}}
\newcommand{\clx}{{\cal X}}

\def\wA{\widetilde{A}}

\def\wH{\widetilde{H}}

\def\a*{{\cal A}_{h,*}}
\def\B{{\cal B}(h)}
\def\B1{{\cal B}_1(h)}
\def\b{{\cal B}^{\rm s.a.}(h)}
\def\b1{{\cal B}^{\rm s.a.}_1(h)}

\newcommand{\ot}{\otimes}

\begin{document}
\begin{center}
\Large{\bf{Quantum Brownian Motion on noncommutative manifolds:}}\\
\Large{\bf {construction, deformation and exit times.}}\\
\vspace{0.15in}
{\large Biswarup Das{ \footnote{Indian Statistical Institute}}}\\
{\large Debashish Goswami {\footnote {Indian Statistical Institute}}}
\end{center}

\vspace{0.15in}

\begin{abstract}
We begin with a review and analytical construction of quantum Gaussian process (and quantum Brownian motions)  in the sense of \cite{schurmann1},\cite{franz1} and others, and then formulate and study in details (with a number of interesting examples) a definition of quantum Brownian motions on those noncommutative manifolds (a la Connes) which are quantum homogeneous spaces of their quantum isometry groups in the sense of \cite{goswami}. We prove that bi-invariant quantum Brownian motion can be `deformed' in a suitable sense. Moreover, we propose a noncommutative analogue of the well-known asymptotics of the exit time of classical Brownian motion. We explicitly analyze such asymptotics for a specific example on noncommutative two-torus $\cla_\theta$, which seems to  behave like a one-dimensional manifold, perhaps reminiscent of the fact that  $\cla_\theta$ is a noncommutative model of  the (locally one-dimensional) `leaf-space' of the Kronecker foliation.
\end{abstract}
\section{Introduction}
There is a very interesting confluence of Riemannian geometry and probability theory in the domain of (classical) stochatistic geometry. The  role of the Brownian motion on a Riemannian manifold cannot be over-estimated in this context; in fact, classical stochastic geometry is almost synonimous with the analysis of Brownian motion on manifolds. Since the inception of the quantum or noncommutative  analogues of Riemannian geometry and the theory of stochastic processes few dacades ago, in the name of noncommutative geometry ( a la Connes) and quantum probability respectively, it has been a natural problem to explore the possibility of interaction and confluence of them.  However, there is not really much work in this direction yet. In \cite{dgkbs}, some case-studies have been made but no general theory was really formulated. The aim of the present paper is to formulate at least some general principle of quantum stochastic geometry using a quantum analogue of Brownian motion on homogeneous spaces. 

The first problem in this context is a suitable noncommutative generalization of Brownian motion, or somewhat more generaly, quantum diffusion or Gaussian processes on manifolds. In the theory of Hudson-Parthasarathy quantum stochastic analysis, a quantum stochastic flow is thought of as  (quantum) diffusion or Gaussian if its quantum stochastic flow equation does not have any `Poisson' or `number' coefficients (see 
\cite{krp}, \cite{dgkbs} and references therein for details). An important question in this context is to characterize the quantum dynamical semigroups which arise as the vacuum expectation semigrooups of quantum Gaussian processes or quantum Brownian motions. In the classical case, such criteria formulated in terms of the `locality' of the generator are quite well-known. However, there is no such intrinsic characterization in the general noncommutative framework, except  a few partial results, e.g. \cite[p.156-160]{dgkbs}, valid only for type I algebras. 

On the other hand, in the algebraic theory of quantum Levy processes  a la Schuermann et al, there are simple and easily verifiable necessary and sufficient conditions for a quantum Levy process on a  bialgebra to be of Gaussian type. This means, in some sense, we have a better understanding of quantum Gaussian processes on quantum groups. On the other hand, for any Riemannian manifold $M$, the group of Riemannian isometries $ISO(M)$ is a Lie group, and Gaussian processes or Brownian motions on the group of isometries induces similar processes on the manifold. For a compact Riemannian manifold the canonical Brownian motion genearted by the (Hodge) Laplacian arises in this way from a bi-invariant Brownian motion on $ISO(M)$. Moreover,  whenever $ISO(M)$ acts transitively on $M$, i.e. when $M$ is a homogeneous space for $ISO(M)$, any covariant Brownian motion does arise from a bi-invariant Brownian motion on $ISO(M)$. All these facts suggest that an extension of the framework of Schuermann et al to quantum homogeneous spaces is called for, and this is indeed one of the objectives of the present article. We also treat these concepts from an analytical viewpoint, realizing quantum Gaussian processes and quantum Brownian motions as bounded operator valued quantum stochastic flows. We then make use of the quantum isometry groups (recently developed by the second author and his collaborators, see, e.g. \cite{goswami,jyoti,dg-banica,jyoti-dg}) of  noncommutative manifolds described by spectral triples and define (and study) quantum Gaussian process or quantum Borwnian motion on those noncommutative manifolds which which are `quantum homogeneous spaces' for their quantum isometry groups.

For constructing interesting noncommutative examples, we investigate the problem of `deforming' quantum Gaussian processes in the framwork of Rieffel (\cite{rieffel}), and prove in particular that any bi-invariant quantum Gaussian process can indeed be deformed. This has helped us to explicitly describe all the Gaussian generators for certain interesting noncommutative manifolds. 
Finally, using our formulation of quantum Brownian motion on noncommutative manifolds,  we propose an analogue of  the classical results about the asymptotics of exit time of Brownian motion from a ball of small volume (see, for example,\cite{pinsky}). We carry it out explicitly for noncommutative two-torus, and obtain quite remarkable results. The asymptotic behaviour in fact differs sharply from the commutative torus, and resembles the asymptotics of a one-dimensional manifold, which is perhaps in agreement with the fact that the noncommutative two-torus is a model for the `leaf space' of the Kronecker foliation, and this `leaf space' is locally (i.e. restricted to a foliation chart) is one dimensional.

\section{Preliminaries}
\subsection{Brownian Motion on Classical Manifolds and Lie-groups}\label{Classical Brownian motion}
Let $M$ be a compact Riemannian manifold of dimension $d,$ equipped with the Riemannian metric $\innerl\cdot,\cdot\innerr.$ Let $Exp_x:T_xM\rightarrow M$ denote the Riemannian exponential map, given by $Exp_x(v)=\gamma(1),$ where $v\in T_xM$ and\\ $\gamma:[0,1]\rightarrow M$ is the geodesic such that $\gamma(0)=x,\gamma^\prime(0)=v.$ The Laplace-Beltrami operator $\Delta$ on $M$ is defined by:
\begin{equation}
\Delta f(x):=\sum_{i=1}^{d}\frac{d^2}{dt^2}f\left(Exp_x(tY_i)\right)|_{_{t=0}}~,
\end{equation}
where $f\in C^2(M)$ and $\{Y_i\}_{i=1}^d$ is a set of complete orthonormal basis of $T_xM.$ This definition is independent of the choice of orthonormal basis of $T_xM.$ If $x_1,x_2,...x_d$ be a local chart at $x,$ then writing 
$\partial_i$ for $\frac{\partial}{\partial x_i},$ $\Delta$ can be written as:
\begin{equation}
\Delta f(x)=\sum_{i,j=1}^d g^{ij}(x)\partial_j\partial_kf(x)-\sum_{i,j,k=1}^dg^{jk}(x)\Gamma^i_{jk}(x)\partial_if(x),
\end{equation}
where $(g^{jk})=(g_{jk})^{-1},$ $g_{jk}(x):=\innerl\partial_j,\partial_k\innerr_x,$ and $\Gamma^i_{jk}$ are the Christoffel symbols.
\bdfn
The Hodge Laplacian on $C^\infty(M)$ is the elliptic differential operator defined in terms of local coordinates $(x_1,x_2,...x_n)$ as:
\begin{equation*}
\Delta_0f=-\frac{1}{\sqrt{det(g)}}\sum_{i,j=1}^{n}\frac{\partial}{\partial x_j}(g^{ij}\sqrt{det(g)}\frac{\partial}{\partial x_i}f),
\end{equation*}
where $f\in C^\infty(M)$ and $g\equiv((g_{ij})).$
\edfn
It may be noted that the Hodge Laplacian on $M$ and the Laplace-Beltrami operator both has similar second order terms and in case $M=\IR^d,$ they coincide. 

It is well known that a standard $d$-dimensional Brownian Motion on $\IR^d$ has the Hodge Laplacian as its generator. An $M$ valued Markov process $X_t^m:(\Sigma,\clf,P)\rightarrow M$ will be called a diffusion process starting at $m\in M$ if $X_0^m=m$ and the generator of the process, say $L,$ when restricted to $C_c^\infty(M)$ will be a second order elliptic differential operator i.e. 
\begin{equation*}
Lf(x)=\sum_{i,j=1}^da_{ij}(x)\partial_i\partial_jf(x)+\sum_{i=1}^db_i(x)\partial_if(x),
\end{equation*}
where $((a_{ij}(\cdot)))$ is a nonsingular positive definite matrix. We will sometimes use the term {\it Gaussian process} for such a Markov process. The diffusion process will be called a Riemannian Brownian motion, if $L$ restricted to $C_c^\infty(M)$ is the Hodge Laplacian restricted to the $C^\infty(M).$ 
\begin{rmrk}
It may be noted that the standard text books e.g. \cite{stroock,mingliao} refer to a Markov process as a Riemmanian Brownian motion if its generator is a Laplace-Belatrami operator. We differ from this usual convention. However our convention will agree with the usual convention in context of symmetric spaces as will be explained later.
\end{rmrk}
The Markov semigroup associated with standard Brownian motion,  given by \\$(T_tf)(m)=\IE(f(X_t^m))$ is called the heat-semigroup. The Brownian Motion gives a ``stochastic dilation" of the heat semigroup.

\paragraph{}
Diffusion processes on classical manifolds are important objects of study as many geometrical invariants can be obtained by analyzing the exit time of the motion from suitably chosen bounded domains. For example,
\begin{ppsn}\cite{pinsky}\label{pinsky}
Consider a hypersurface $M\subseteq\IR^d$ with the Brownian motion process $X_t^m$ starting at $m.$ Let $T_\varepsilon=inf\{t>0:$ $\|X_t^m-m\|=\varepsilon\}$ be the exit time of the motion form an extrinsic ball of radius $\varepsilon$ around $m.$ Then we have $$\IE_m(T_\varepsilon)=\varepsilon^2/2(d-1)+\varepsilon^4H^2/ 8(d+1)+O(\varepsilon^5),$$ where $H$ is the mean curvature of $M.$ 
\end{ppsn}
\begin{ppsn}\cite{Liao}
Let $M$ be an $n$-dimensional Riemannian manifold with the distance function $d(\cdot,\cdot),$ and $X_t^x$ be the Brownian motion starting at $x\in M.$ Let $\rho_t:=d(x,X_t^x)$ (known as the radial part of $X_t^x$). Let $T_\epsilon$ be the first exit time of $X_t^x$ form a ball of radius $\epsilon$ around $x,$ $\epsilon$ being fixed. Then $$\IE(\rho_{t\wedge T_\epsilon}^2)=nt-\frac{1}{6}S(x)t^2+o(t^2),$$ where $S(x)$ is the scalar curvature at $x.$
\end{ppsn}
We shall need a slightly modified version of the asymptotics described by Proposition \ref{pinsky}, using the expression obtained in \cite{gray}, of the volume of a small extrinsic ball as described below:

Let $V_m(\epsilon)$ denote the ball of radius $\epsilon$ around $m\in M.$ Let $n$ be the intrinsic dimension of the manifold. Then we have
\begin{equation}
V_m(\epsilon)=\frac{\alpha_n\epsilon^n}{n}\left(1-K_1\epsilon^2+K_2\epsilon^4+O(\epsilon^6)\right)_m,
\end{equation}
where $\alpha_n:=2\Gamma(\frac{1}{2})^n\Gamma(\frac{n}{2})^{-1}$ and $K_1,K_2$ are constants depending on the manifold. 

The intrinsic dimension $n$ of the hypersurface $M$ is obtained from $\IE(\tau_\epsilon)$ as the unique integer $n$ satisfying 
$\lim_{\epsilon\rightarrow0}\frac{\IE(\tau_{_{\epsilon}})}{V_{_{\epsilon}}^{\frac{2}{m}}}=
\begin{cases}
\infty~\mbox{if $m$ is just less than $n$}\\
\neq0~\mbox{if $m\neq n$}\\
=0~\mbox{if $m>n.$}
\end{cases}
$

Observe that $\frac{V(\epsilon)^{\frac{2}{n}}}{\epsilon^2}\rightarrow (\frac{\alpha_n}{n})^{\frac{2}{n}}$ and 
$\frac{V(\epsilon)^{\frac{4}{n}}}{\epsilon^4}\rightarrow(\frac{\alpha_n}{n})^{\frac{4}{n}}$ as $\epsilon\rightarrow0^+.$ So the asymptotic expression of Proposition \ref{pinsky} can be recast as 
\begin{equation*}
\IE(\tau_\epsilon)=\frac{1}{2(d-1)}(\frac{V(\epsilon)n}{\alpha_n})^{\frac{2}{n}}+\frac{H^2}{8(d+1)}(\frac{V(\epsilon)n}{\alpha_n})^{\frac{4}{n}}+O(V(\epsilon)^{\frac{5}{n}}).
\end{equation*}
In particular, we get the extrinsic dimension $d$ and the mean curvature $H$ by the following formulae:
\begin{equation}\label{intrinsic dimension}
d=\frac{1}{2}(1+\lim_{\epsilon\rightarrow0}\frac{1}{\IE(\tau_\epsilon)}(\frac{nV(\epsilon)}{\alpha_n})^{\frac{2}{n}}),
\end{equation}
\begin{equation}\label{formula for mean curvature}
H^2=8(d+1)(\frac{\alpha_n}{n})^{\frac{4}{n}}\lim_{\epsilon\rightarrow0}\frac{\IE(\tau_\epsilon)-\frac{1}{2(d-1)}(\frac{nV(\epsilon)}{\alpha_n})^{\frac{2}{n}}}{V(\epsilon)^{\frac{4}{n}}}.
\end{equation}

\paragraph{}
If there is a Lie group $G$ which has a left (right) action on $M,$ then it is natural to study the diffusion processes $X_t\equiv\{X_t^m,m\in M\}$ which are left (right) invariant in the sense that\\ $g\cdot X^m_t=X_t^{g\cdot m}~(X_t^m\cdot g=X^{m\cdot g}_t)$ almost everywhere for all $g\in G,m\in M.$ In particular, if $M=G,$  we shall call $X_t^e$ (where $e$ is the identity element of $G$) the cannonical left (right) invariant diffusion  process, and we will usually drop the adjective left or right. For such a diffusion  process, the generator $L=\sum_iA_iX_i+\frac{1}{2}\sum_{i,j}B_{ij}X_iX_j,$ where $(B_{ij})_{i,j}$ is a non-negetive definite matrix and $\{X_1,...X_d\}$ is a basis of the Lie-algebra $\clg.$ The diffusion  process defined above is called bi-invariant if it is both left and right invariant. We also note that such processes constitute a special class of the so-called {\it Levy process} on groups \cite{mingliao} i.e. a stochastic process which has almost surely cadlag paths, left (right) independent increment and left (right) stationary increments (see \cite{mingliao} for details).
\begin{ppsn}(\cite{ito})
A necessary and sufficient condition for a Diffusion process Motion on a Lie group $G$ to be bi-invariant is the following: 
$$A_jC_{kj}^l=0,~B_{ij}C_{kj}^l+B_{jl}C_{kj}^i=0~~(1\leq i,k,l\leq d),$$ where $C^l_{kj}$ are the Cartan coefficients of $G$.
\end{ppsn}
If $M$ is a symmetric space (i.e. the isometry group $G$ acts transitively on $M$), it is interesting to study the diffusion processes on $M$ which are covariant i.e. 
$\alpha_g\circ L=L\circ\alpha_g$ for all $g\in G,$ where $L$ is the generator of the diffusion process and $\alpha:G\times M\rightarrow M$ is the action of $G$ on $M.$ 
\begin{ppsn}\label{ming2}\cite{mingliao}
Let $G$ be a Lie group and let $K$ be a compact subgroup. If $g_t$ is a right $K$ invariant left Levy process in $G$ with $g_0=e$, then its one point motion from $o=eK$ in $M=G/K$ is a $G$ invariant Feller process in $M.$ Conversely, if $x_t$ is a $G$ invariant Feller process in $M$ with $x_0=o$, then there is a right $K$ invariant left Levy process $g_t$ in $G$ with $g_0=e$ such that its one-point motion in $M$ from $o$ is identical to the process $x_t$ in distribution.
\end{ppsn}
Suppose that $G$ is compact. The proof of Proposition \ref{ming2}, as in \cite{mingliao} then  implies that any covariant diffusion process $x_t$ on $M$ can be realized as restriction of a corresponding right $K$ invariant diffusion process on $G.$

\subsection{Quantum Stochastic Calculus}\label{Quantum stochastics}
We refer the reader to \cite{krp} and \cite{dgkbs} for the basics of Hudson-Parthasarathy (H-P for short) formalism and Evans-Hudson (E-H) formalism of quantum stochastic calculus which we briefly review here.\\Let $\clv,\clw$ be vector spaces and $\clv_0\subseteq\clv,~\clw_0\subseteq\clw$ be vector subspaces. We will denote by $Lin(\clv_0,\clw_0)$ the space of all linear maps with domain $\clv_0$ and range lying inside $\clw_0.$\\ For 
$T\in Lin(\clv_0\ot\clw_0,\clv\ot\clw),~\xi,\eta\in\clw_0,$ denote by $\innerl\xi,T_\eta\innerr$ the map from $\clv_0$ to $\clv$ defined by $\innerl u,\innerl\xi,T_\eta\innerr v\innerr=\innerl u\ot\xi,T(v\ot\eta)\innerr,$ for $u,v\in\clv_0.$ Furthermore for $L\in Lin(\clv_0,\clv\ot\clw),$ denote by $\innerl\xi,L\innerr$ the operator defined by $\innerl\innerl\xi,L\innerr u,v\innerr=\innerl L(u),v\ot\xi\innerr$ and define $\innerl L,\xi\innerr:=\innerl\xi,L\innerr^*,$ whenever it exists. For a Hilbert space $\clh,$ $\Gamma(\clh)$ will denote the symmetric fock space over $\clh$ and for $f\in\clh,$ $e(f)$ will denote the exponential vector on $f$ (see \cite{dgkbs}).
\subsubsection{Hudson-Parthasarathy equation}
Let $h$ and $k_0$ be Hilbert spaces with subspaces $\clv_0\subseteq h $ and $\clw_0\subseteq k_0.$  Consider a quadruple of operators $(A,R,S,T),$ such that $D(R),D(S)$ and $D(A)$ are subspaces of $h,$ $D(T)$ is a subspace of $h\ot k_0.$ Suppose $A\in Lin(D(A),h),~R\in Lin(D(R),h\ot k_0),$ 
$S\in Lin(D(S),h\ot k_0)$ and $T\in Lin(D(T),h\ot k_0).$ Furthermore assume that 
$$\clv_0\subseteq\bigcap_{\xi,\eta\in\clw_0}\left(D(\innerl\xi,R\innerr)\cap D(\innerl S,\eta\innerr)\cap D(\innerl\xi,T_\eta\innerr)\cap D(A)\right).$$ Let $\Gamma:=\Gamma(L^2(\IR_+,k_0)).$ A family of operators $(V_t)_{t\geq0}\in Lin(h\ot\Gamma,h\ot\Gamma)$ is called a solution of an H-P equation with initial Hilbert space $h,$ noise space $k_0$ and initial condition $V_0=id,$ if it satisfies an equation of the form
$$\innerl ue(f),V_tve(g)\innerr=\innerl ue(f),ve(g)\innerr+\int_0^t\innerl ue(f),V_s\circ\left(A+\innerl f(s),S\innerr+\innerl R,g(s)\innerr+\innerl f(s),T_{g(s)}\innerr\right)ve(g)\innerr ds,$$ for $u,v\in\clv_0,$ $f,g$ being step functions taking values in $\clw_0.$ We will denote the above equation symbolically by 
\begin{equation}
\begin{split}
dV_t=V_t\circ(a^\dagger_S(dt)+&a_R(dt)+\Lambda_T(dt)+Adt),\\ 
&V_0=id.
\end{split}
\end{equation}
\subsubsection{Evans-Hudson equation}
Let $\cla\subseteq B(h)$ be a $C^*$ or von-Neumann algebra such that there exists a dense (in the appropriate topology) $\ast$-subalgebra $\cla_0.$ Suppose $\cll$ is a densely defined map on $\cla$ such that $\cla_0\subseteq D(\cll).$ Assume the following:
\begin{enumerate}
\item There exists a $\ast$-representation $\pi:\cla\rightarrow\cla\ot B(k_0)$ (normal in case $\cla$ is a von-Neumann algebra);
\item a $\pi$-derivation $\delta:\cla_0\rightarrow\cla\ot k_0,$ such that :
$$\delta(x)^*\delta(y)=\cll(x^*y)-x^*\cll(y)-\cll(x^*)y,$$ for all $x,y\in\cla_0.$
\end{enumerate}
Let $\sigma:=\pi-id_{B(h\ot k_0)}.$ Then a family of $\ast$-homomorphism $j_t:\cla\rightarrow\cla^{\prime\prime}\ot B(\Gamma),t\geq0$ is said to satisfy an E-H equation with initial condition $j_0=id,$ if 
\begin{equation}
\begin{split}
\innerl ue(f),j_t(x)ve(g)\innerr=\innerl ue(f),xve(g)\innerr
&+\int_0^t\langle ue(f),j_s\circ(\cll(x)+\innerl f(s),\delta(x)\innerr
+\innerl\delta(x^*),g(s)\innerr\\
&+\innerl f(s),\sigma(x)_{g(s)}\innerr)ve(g)\rangle ds,
\end{split}
\end{equation}
 for all $x\in\cla_0$ $u,v\in h,$ $f,g$ being step functions. We will write the above equation symbolically as
\begin{equation}
\begin{split}
dj_t=j_t\circ(a^\dagger_\delta(dt)+&a_{\delta^\dagger}(dt)+\Lambda_\sigma(dt)+\cll dt)\\
&j_0=id. 
\end{split}
\end{equation}
We will call  any of the two equations described above a quantum stochastic differential equation (QSDE for short). It is well known that solutions of such QSDE are cocycles (see \cite{dgkbs,krp}). If the solution of an H-P equation is unitary, then the solution will be called the H-P dilation of the vacuum semigroup $T_t(x):=\innerl e(0),V_t(x\ot1_{\Gamma})V_t^*{e(0)}\innerr$ and $j_t$ is called an E-H dilation of the vacuum semigroup $T_t(x):=\innerl e(0),j_t(x)e(0)\innerr.$\\
\bdfn
A semigroup $(T_t)_{t\geq0}:\cla\rightarrow\cla$ is called a quantum dynamical semigroup (QDS for short) if for each $t,$ $T_t$ is a contractive completely positive map on $\cla$ (normal in case $\cla$ is a von-Neumann algebra). The semigroup is callled conservative if $T_t(1)=1$ for all $t\geq0.$
\edfn
Typical examples of QDS are the Markov semigroups associated with a Markov process as well as the vacuum semigroups described above.
\bthm\label{THE THEOREM}
Let $R:D(R)\rightarrow h\ot k_0$ be a densely defined closed operator with $D(R)\subseteq h,$ for Hilbert spaces $h,k_0.$ Suppose there exists a dense subspace $\clw_0\subseteq k_0,$ such that $u\ot\xi\in D(R^*)$ for $u\in D(R),\xi\in\clw_0.$ Let $H$ be a densely defined self adjoint operator on $h$ such that $iH-\frac{1}{2}R^*R~(=G)$ as well as $-iH-\frac{1}{2}R^*R~(=G^*)$ generate $C_0$ semigroups in $h.$ Furthermore, suppose that both of $D(G)$ and $D(G^*)$ are contained in $D(R).$  Then the QSDE:
\begin{equation}
\begin{split}
dU_t=U_t\circ(a^\dagger_R(dt)&-a_R(dt)+(iH-\frac{1}{2}R^*R)dt)\\
&U_0=id; 
\end{split}
\end{equation}
has a unique solution which is unitary.
\ethm
\begin{proof}
Let  $Z=$
$\begin{pmatrix}
iH-\frac{1}{2}R^*R&R^*\\
R&0
\end{pmatrix}.$ Suppose $H=u|H|,~R=v|R|$ be the polar decomposition of $H$ and $R$ respectively.\\ 
Put $A^{(n)}:=iu(1+\frac{|H|}{n})^{-1}|H|-\frac{1}{2}R^{(n)*}R^{(n)},$ where $R^{(n)}:=R(1+\frac{|R|}{n})^{-1},$ and  
$Z^{(n)}=
\begin{pmatrix}
A^{(n)}&R^{(n)*}\\
R^{(n)}&0
\end{pmatrix}.
$
Then it can be verified that all the conditions of Theorem 7.2.1 in page 174 of \cite{dgkbs} hold. Thus the above equation has a contractive solution $U_t,t\geq0.$ Now observe that in the notation of Theorem 7.2.3 in page 179 of \cite{dgkbs}, $\cll^\gamma_\eta(I)=0,$ for all $\gamma,\eta\in\IC\oplus\clw_0.$ We will prove that $\beta_\lambda=\{0\}.$ Formally define $L(x)=R^*(x\ot1_{k_0})R+xG+G^*x,$ where $G=iH-\frac{1}{2}R^*R.$ Then the conditions ${\bf (Ai)}$ and ${\bf (Aii)}$ in page 39 of \cite{dgkbs} hold. So by Theorem 3.2.13 in page 46 of \cite{dgkbs}, there is a minimal semigroup say $(\wtil{T_t})_{t\geq0}$ on $B(h),$ whose form generator $\wtil{\cll^{min}}$ on a certain dense subspace is  of the form $R^*(x\ot1_{k_0})R+xG+G^*x.$ We prove that $\wtil{T_t}$ is conservative:\\
Let $\cld\subseteq h$ be the subspace such that for $x\in\cld,$ $R^*(x\ot1_{k_0})R+xG+G^*x\in B(h),$ i.e. $L(x)\in B(h)$ for $x\in\cld.$ Note that $1:=1_{B(h)}\in\cld.$ Let $(\wtil{T}_{\ast,t})_{t\geq0}$ be the predual semigroup of $(\wtil{T_t})_{t\geq0}.$ It is known (see chapter 3 of \cite{dgkbs}) that for $\sigma\in B_1(h)~(B_1(h)$ is the space of trace class operators on $h),$ the linear span of operators $\rho$ of the form $\rho=(1-G)^{-1}\sigma(1-G)^{-1},$ denoted by $\clb$ say, belongs to $D(\wtil{\cll^{min}_\ast}),$ $\wtil{\cll^{min}_\ast}$ being the generator of $(\wtil{T}_{\ast,t})_{t\geq0}.$ Moreover we have $\wtil{\cll^{min}_\ast}(\rho)=R\rho R^*+G\rho+\rho G^*$ for $\rho\in\clb,$ and $\clb$ is a core for $\wtil{\cll^{min}_\ast}.$    Now for $a\in \cld,\rho\in \clb,$ $tr(L(a)\rho)=tr(a\wtil{\cll^{min}_\ast}(\rho)).$ Since $\clb$ is a core for $\wtil{\cll^{min}_\ast},$ we have $tr(L(a)\rho)=tr(a\wtil{\cll^{min}_\ast}(\rho))$ for all $\rho\in D(\wtil{\cll^{min}_\ast}).$ Observe that for $\rho\in D(\wtil{\cll^{min}_\ast}),$
\begin{equation}
\begin{split}
tr\left(\frac{\wtil{T}_t(a)-a}{t}\rho\right)&=
tr\left(a\left(\frac{\wtil{T}_{\ast,t}(\rho)-\rho}{t}\right)\right)=
tr\left(a\wtil{\cll^{min}_\ast}\left(t^{-1}\int_0^t\wtil{T}_{\ast,s}(\rho)ds\right)\right)\\
&=tr\left(L(a)\left(t^{-1}\int_0^t\wtil{T}_{\ast,s}(\rho)ds\right)\right);
\end{split}
\end{equation}
which proves that $a\in D(\wtil{\cll^{min}})$ and by continuity, $L(a)=\wtil{\cll^{min}}(a),$ for all $a\in\cld.$ Now $L(1)=0$ which implies that $\wtil{\cll^{min}}(1)=0,$ i.e. $(\wtil{T_t})_{t\geq0}$ is conservative. Thus by condition $(v)$ in page 48 of Theorem 3.2.16 of \cite{dgkbs}, $\beta_\lambda=\{0\}.$ The same set of arguments hold for 
$G=-iH-\frac{1}{2}R^*R,$ which implies that $\wtil{\beta}_\lambda=\{0\}$ (in the notation of Theorem 7.2.3 of \cite{dgkbs}). Moreover $\wtil{\cll}^\gamma_\eta(I)=0.$ Thus all the conditions of Theorem 7.2.3 in page 179 of \cite{dgkbs} hold, which proves that the solution is unitary. The uniqueness follows from (\cite[Theorem  p. ]{mohari}).
\end{proof}

\subsection{Compact Quantum Group}\label{CQG}
We shall refer the reader to \cite{CQG} and the references therein for the basics of Compact Quantum Group, which we briefly review here.  Suppose $\cla,\clb$ are $\ast$-algebras. If both $\cla$ and $\clb$ are $C^*$ algebras, then  $\cla\ot\clb,$ will mean the injective tensor product, otherwise they will denote the algebraic tensor product.
\bdfn 
A compact quantum group (CQG for short) is a unital $C^\ast$-algebra $\clq\subseteq B(k)$ equipped with a unital $\ast$-homomorphism (called coproduct) $\Delta:\clq\rightarrow\clq\ot\clq$ such that $\Delta(\clq)(\clq\ot1)$ as well as $\Delta(\clq)(1\ot\clq)$ are dense in $\clq\ot\clq.$
\edfn

Given a CQG $\clq,$ there exists a unique state $h$ on $\clq$ called the Haar state (see \cite{CQG}) satisfying 
$(h\ot id)\Delta(a)=(id\ot h)\Delta(a)=h(a){\bf 1_{\clq}}.$ Moreover, we have a dense $\ast$-algebra $\clq_0\subseteq\clq$ which is a Hopf$\ast$-algebra, equipped with counit $\epsilon:\clq_0\rightarrow\IC$ and antipode $\kappa:\clq_0\rightarrow\clq_0,$ satisfying $(\epsilon\ot id)\Delta=(id\ot\epsilon)\Delta=id.$ Throughout the discussion, we will use Sweedler's notation for CQG i.e. $\Delta(a)=a_{(1)}\ot a_{(2)},$ for all $a$ in $\clq_0.$

For a map $X\in B(\clh_1\ot\clh_2),$ we will use the notation $X_{(12)}$ to denote the operator $X\ot I_{\clh_3}$ and the notation $X_{(13)}$ to denote the operator $\Sigma_{23}X_{(12)}\Sigma_{23},$ where $\Sigma_{23}\in U(\clh_1\ot\clh_2\ot\clh_3)$ is the flip between $\clh_2$ and $\clh_3.$ 
\bdfn 
A map $U:\clh\rightarrow\clh\ot\clq,$ where $\clh$ is a hilbert space is called an unitary (co)representation of the CQG $\clq$ on the Hilbert space $\clh,$ if  $\wtil{U}\in\clm(\clk(\clh)\ot \clq)$ defined by $\wtil{U}(\xi\ot b):=U(\xi)(1\ot b),$ for $\xi\in\clh$ and $b\in\clq,$ is an unitary operator which further satisfies $(id_\clh\ot\Delta)\wtil{U}=\wtil{U}_{(12)}\wtil{U}_{(13)}.$
\edfn
If dimension of $dim\clh=n<\infty,$ we may alternatively represent $U$ by the $\clq$-valued $n\times n$ invertible matrix $((\innerl Ue_i,e_j\innerr))_{i,j},$ where $\{e_k\}_{k=1}^n$ is an orthonormal basis of $\clh.$ We will call $n$ the dimension of the representation $U.$

By G.N.S construction using $h,$ let $\clq\subseteq B(L^2(h)).$ Then $\Delta$ viewed as $\Delta:L^2(h)\rightarrow L^2(h)\ot\clq$ becomes an unitary representaion (say $U$) such that $\Delta(x)=U(x\ot 1)U^*.$ Moreover, $\clq_0$ is the linear span of the matrix coefficients of all finite dimensional unitary inequivalent corepresentations (see \cite{jyoti,CQG}). Furthermore,
$L^2(h)=\oplus_{\pi}\clh_\pi$ and $\clq_0(\subseteq L^2(h))=\oplus^{alg}_{\pi}\clh_\pi$ where $\clh_\pi$ is a finite dimensional vector space of dimension $d_\pi^2$ obtained from the decomposition of $\Delta$ (viewed as $U$) into finite dimensional irreducibles $\pi$ of dimension $d_\pi$ by the Peter-Weyl theory for CQG\cite{CQG}.

\subsubsection{Action of a compact quantum group on a $C^\ast$-algebra}\label{3-spaces}
We say that a CQG $(\clq,\Delta)$ (co)-acts on a unital $C^\ast$-algebra $\clb,$ if there is a unital $C^*$ homomorphism (called an action) $\alpha:\clb\rightarrow\clb\ot\clq,$ satisfying the following: 
\begin{enumerate}
\item $(\alpha\ot id)\circ\alpha=(id\ot\Delta)\circ\alpha,$
\item the linear span of $\alpha(\clb)(1\ot\clq)$ is dense in $\clb\ot\clq.$ 
\end{enumerate}
It has been shown in \cite{podles} that (2) is equivalent to the existence of a dense $\ast$-subalgebra $\clb_0\subseteq\clb$ such that $\alpha(\clb_0)\subseteq\clb_0\ota\clq_0.$
We say that an action $\alpha$ is faithful, if there is no proper Woronowicz $C^*$-subalgebra (see \cite{jyoti},\cite{CQG}) $\clq_1$ of $\clq$ such that $\alpha$ is a $C^*$ action of $\clq_1$ on $\clb.$ 
We refer the reader to \cite{jyoti} and the references therein for details of $C^*$ action.

For a CQG $(\clq,\Delta),$ denote by $Irr_{_{\clq}},$ the set of inequivalent, unitary irrereducible representations of $\clq$ and let $u^\gamma$ be a co-representation of $\clq$ of dimension $d_\gamma,$ for $\gamma\in Irr_{_{\clq}}.$ We will call a vector subspace $V\subseteq\clb$ a subspace correponding to $u^\gamma$ if 
\begin{itemize}
\item $dimV=d_\gamma,$ 
\item $\alpha(e_i)=\sum_{k=1}^{d_\gamma} e_k\ot u^\gamma_{ki},$ for some orthonormal basis $\{e_j\}_{j=1}^{d_\gamma}$ of $V.$ 
\end{itemize}
\begin{ppsn}\cite{podles}\label{spectral subspaces of action}
Let $\alpha$ be an action of a CQG $(\clq,\Delta)$ on a $C^\ast$-algebra $\clb.$ Then there exists vector subspaces 
$\{W_{\gamma}\}_{\gamma\in Irr_{_{\clq}}}$ of $\clb$ such that 
\begin{enumerate}
\item $\clb=\overline{\oplus_{\gamma\in Irr_{_{\clq}}}W_\gamma}$
\item For each $\gamma\in Irr_{_{\clq}},$ there exists a set $I_\gamma$ and vector subspaces $W_{\gamma i},~i\in I_\gamma,$ such that 
\begin{enumerate}
\item[a.] $W_\gamma=\oplus_{i\in I_\gamma}W_{\gamma i}.$
\item[b.] $W_{\gamma i}$ corresponds to $u^\gamma$ for each $i\in I_\gamma.$ 
\end{enumerate}
\item Each vector subspace $V\subseteq\clb$ corresponding to $u^\gamma$ is contained in $W_\gamma.$
\item The cardinal number of $I_\gamma$ doesn't depend on the choice of $\{W_{\gamma i}\}_{i\in I_\gamma}.$ It is denoted by $c_\gamma$ and called the multiplicity of $u^\gamma$ in the spectrum of $\alpha.$
\end{enumerate}
\end{ppsn}
\bdfn
A CQG $(\clq^\prime,\Delta^\prime)$ is called a quantum subgroup of another CQG $(\clq,\Delta)$ if there is a Woronowicz $C^*$-ideal $\clj$ of $\clq$ such that 
$(\clq^\prime,\Delta^\prime)\cong(\clq,\Delta)/\clj.$
\edfn
\bdfn\cite{podles}\label{podles}
Suppose a CQG $(\clq,\Delta)$ acts on a $C^*$-algebra $\clb.$ Then $\clb$ is called 
\begin{enumerate}
\item {\bf A quotient of $(\clq,\Delta)$ by a quantum subgroup $(S,\Delta|_{_S})$} if:
\begin{enumerate}
\item[a)] $\clb$ is $C^*$-isomorphic to the algebra $\clc:=\{x\in\clq:(\pi\ot id)\Delta(x)=1\ot x\},$
\item[b)] the action $\alpha$ is given by $\alpha:=\Delta|_{\clc},$
\end{enumerate}
where $\pi$ is the CQG morphism from $\clq$ to $S.$ 
\item {\bf Embeddable}, if there exists a faithful $C^*$-homomorphism $\psi:\clb\rightarrow\clq$ such that\\ $\Delta\circ\psi=(\psi\ot id)\circ\alpha.$
\item {\bf Homogeneous} if the multiplicity of the trivial representation of $\clq$ in the spectrum of $\alpha$\\ (see \cite{podles}) be $1$. 
\end{enumerate}
\edfn
Henceforth, we will refer to $\clb$ as a quantum space. It can be easily shown that a quantum space is homogeneous if and only if the corresponding action is ergodic (i.e. $\alpha(x)=x\ot I$ implies $x$ is a scalar multiple of the identity of $\clb$.

\begin{ppsn}\cite{podles}
Let $\alpha$ be the action of a CQG $(\clq,\Delta)$ on a $C^*$-algebra $\clb.$ Then 
\begin{enumerate}
\item[a)] $(\clb,\alpha)$ is quotient $\Rightarrow(\clb,\alpha)$ is embeddable $\Rightarrow(\clb,\alpha)$ is homogeneous.
\item[b)] In the classical case, $(\clb,\alpha)$ is quotient $\Leftarrow\Rightarrow(\clb,\alpha)$ is embeddable $\Leftarrow\Rightarrow(\clb,\alpha)$ is homogeneous.   
\end{enumerate}
\end{ppsn}
We refer the reader to \cite{podles} for more discussions on these three types of quantum spaces.

\bdfn
For a linear map $\clp:\clq_0\rightarrow\clb,$ where $\clb$ is a $\ast$-algebra, define $\wtil{\clp}:\clq_0\rightarrow\clq_0\ot\clb$ by $\wtil{\clp}:=(id\ot\clp)\circ\Delta.$ For two such maps $\clp_1,\clp_2,$ define $\clp_1\ast\clp_2:=m_\clb\circ(\clp_1\ot\clp_2)\circ\Delta,$ where $m_\clb$ denotes the multiplication in $\clb.$
\edfn 
It follows  that $(id_{\clq_0}\ot m_\clb)\circ(\wtil{\clp_1}\ot id_\clb)\circ\wtil{\clp_2}=\wtil{\clp_1\ast\clp_2}.$ Observe that 
$(id\ot\widetilde{\clp})\Delta=\Delta\circ\widetilde{\clp}.$

\subsubsection{Rieffel Deformation}\label{deformation defined by reiffel}
Let $\theta=((\theta_{kl}))$ be a skew symmetric matrix of order $n.$ We denote by $C^*(\IT^n_\theta)$ the universal $C^*$-algebra generated by $n$ unitaries $(U_1,U_2,...U_n)$ satisfying $U_kU_l=e^{2\pi\theta_{kl}}U_lU_k,$ for $k\neq l.$ If $\theta_{kl}=\theta_0$ for 
$k<l,$ where $\theta_0\in\IR,$ we will denote the corresponding universal $C^*$-algebra by $C^*(\IT^n_{\theta_0})$ and $\clw$ will denote the $\ast$-subalgebra generated by  unitaries $U_1,U_2,...U_n.$

\paragraph{}
Let $\cla$ be a unital $C^*$-algebra on which there is a strongly continuous $\ast$-automorphic action $\sigma$ of $\IT^n.$ Denote by $\tau$ the natural action of $\IT^n$ on $C^*(\IT^n_\theta)$ given on the generators $U_i's$ by $\tau(\overline{z})U_i=z_iU_i,$ where $\overline{z}=(z_1,z_2,...z_n)\in\IT^n.$ Let $\tau^{-1}$ denote the inverse action $s\rightarrow\tau_{-s}.$
We refer the reader to \cite{rieffel} for the original approach of Rieffel using twisted convolution.
\bdfn
The fixed point algebra of $\cla\ot C^*(\IT^n_\theta),$ under the action $(\sigma\times\tau^{-1}),$ i.e.
$(\cla\ot C^*(\IT^n_\theta))^{\sigma\times\tau^{-1}},$ is called the Rieffel deformation of $\cla$ under the action $\sigma$ of $\IT^n,$ and is denoted by $\cla_\theta.$
\edfn
There is a natural isomorphism between $(\cla_\theta)_{-\theta}$ and $\cla,$ given by the identification of $\cla$ with the subalgebra of $(\cla\ot C^*(\IT^{n}_\theta)\ot C^*(\IT^n_\theta))^{(\sigma\ot id)\times \tau^{-1}}$ generated by elements of the form $a_{\overline{p}}\ot U^{\overline{p}}\ot( U^{\prime})^{\overline{p}},$ where $\overline{p}=(p_1,p_2,...p_n)\in\IZ^n,$ $U^{\overline{p}}:=U_1^{p_1}U_2^{p_2}...U_n^{p_n},$ $(U^\prime)^{\overline{p}}:=(U^\prime_1)^{p_1}(U^\prime_2)^{p_2}....(U_n^\prime)^{p_n},$ $U_1^\prime,U_2^\prime,...U_n^\prime$ being the generators of $C^*(\IT^n_{-\theta})$ and $a_{\overline{p}}$ belongs to the spectral subspace of the action $\sigma$ corresponding to the character $\overline{p}.$ 

Let $\clq$ be a CQG with coproduct $\Delta$ and assume that there exists a surjective CQG morphism $\pi:\clq\rightarrow C(\IT^n)$ which identifies $C(\IT^n)$ as a quantum subgroup of $\clq.$ For $s\in\IT^n,$ let $\Omega(s)$ denote the state defined by $\Omega(s):=ev_s\circ\pi,$ where $ev_s$ denotes evaluation at $s.$ Define an action of $\IT^{2n}$ on $\clq$ by $(s,u)\rightarrow \chi_{(s,u)},$ where $\chi_{(s,u)}:=(\Omega(s)\ot id)\circ\Delta\circ(id\ot\Omega(-u))\circ\Delta.$ It has been shown in \cite{wang} that the Rieffel deformation $\clq_{\widetilde{\theta},-\widetilde{\theta}}$ of $\clq$ with respect to 
$\widetilde{\theta}:=\begin{pmatrix}
 0&\theta\\
-\theta&0 
 \end{pmatrix}
$ can be given a unique CQG structure such that the such that the Hopf$\ast$ algebra of $\clq_{\theta,-\theta}$ is isomorphic as a coalgebra with the cannonical Hopf$\ast$ algebra of $\clq.$ 

\subsubsection{Quantum Isometry group}\label{QISO} We begin by defining spectral triple (also called spectral data). We shall refer the reader to \cite{connes} and \cite{jyoti} for details.
\bdfn
An odd spectral triple or spectral data is a triple $(\cla^\infty,H,\cld)$ where $H$
is a separable Hilbert space, $\cla^\infty$ is a $*$-subalgebra of $B(H)$,(not necessarily norm closed) and $\cld$ is a self adjoint (typically unbounded) operator such that for all $a$ in $\cla^\infty,$ the operator $[\cld,a]$ has a bounded extension. Such a spectral triple is also called an odd spectral triple. If in addition, we have $\gamma$ in $B(H)$ satisfying $\gamma=\gamma^∗=\gamma^{−1},~\cld\gamma=−\gamma\cld$ and $[a,\gamma]=0$ for all $a$ in $\cla^\infty,$ then we say that the quadruplet $(\cla^\infty,H,\cld,\gamma)$ is an even spectral triple. The operator $\cld$ is called the Dirac operator corresponding to the spectral triple.
\edfn
Since in the classical case, the Dirac operator has compact resolvent if the manifold is compact, we say that the spectral triple is of compact type if $\cla^\infty$ is unital and $\cld$ has compact resolvent. A spectral triple $(\cla^\infty,H,\cld)$ will be called $\Theta$ summable if $e^{-t\cld^2}$ is a trace class operator ($t\geq0$). Next we discuss the notion of Hilbert space of $k$-forms in non-commutative geometry.
\begin{ppsn}Given an algebra $\clb,$ there is a (unique upto isomorphism)$\clb-\clb$ bimodule $\Omega^1(\clb)$ and a derivation $\delta:\clb\rightarrow\Omega^1(\clb),$ satisfying the following properties:
\begin{enumerate}
 \item $\Omega^1(\clb)$ is spanned as a vector space by elements of the form $a\delta(b)$ with $a$, $b$ belonging to $\clb$; and
\item for any $\clb-\clb$ bimodule $E$ and a derivation $d:\clb\rightarrow E,$ there is an unique $\clb-\clb$ linear map $\eta:\Omega^1(\clb)\rightarrow E$ such that $d=\eta\circ\delta.$
\end{enumerate}
\end{ppsn}
The bimodule $\Omega^1(\clb)$ is called the space of universal 1-forms an $\clb$ and $\delta$ is called the universal derivation. Given a $\Theta$-summable spectral triple, $(\cla^\infty,H,\cld),$ it is possible to define an inner product structure on $\Omega^0(\cla^\infty)\equiv\cla^\infty$ and $\Omega^1(\cla^\infty).$ The corresponding Hilbert spaces are are denoted by $\clh^0_\cld$ and $\clh^1_\cld$ and are called the Hilbert spaces of zero and one forms respectively (see \cite{jyoti}).
 
\paragraph{}
We now define quantum isometry group.
Let $(\cla^\infty,H,\cld)$ be a $\Theta$-summable spectral triple which is admissible in the sense that it satisfies the regularity conditions (i)-(v) as given in \cite[pages 9-10]{goswami}.\\ Let $\cll:=-d_{\cld}^*d_{\cld },$ which is a densely defined self-adjoint operator on $\clh_0$ and is called the Laplacian of the spectral triple.
We will denote by $\IQ^{\prime,\cll}$ the category whose objects are triplets $(S,\Delta,\alpha)$ where $(S,\Delta)$ is a CQG acting smoothly and isometrically on the given noncommutative manifold, with $\alpha$ being the corresponding action.
\begin{ppsn}\label{DG}\cite{goswami}
For any admissible spectral triple $(\cla^\infty,H,\cld),$ the category $\IQ^{\prime,\cll}$ has a universal object denoted by $(QISO^\cll,\alpha_0).$ Moreover, $QISO^\cll$ has a coproduct $\Delta_0$ such that $(QISO^\cll,\Delta_0)$ is a CQG and $(QISO^\cll,\Delta_0,\alpha_0)$ is a universal object in the category $\IQ^{\prime,\cll}.$ The action $\alpha_0$ is faithful.
\end{ppsn}
The reader may see \cite{goswami} and \cite{jyoti} for further details of $QISO^\cll.$ We now give some examples of quantum isometry groups. 
\begin{enumerate}
\item {\bf non-commutative $2$-tori:} The non-commutative $2$-tori $C^*(\IT^2_\theta)$ is the universal $C^*$-algebra generated by a pair of unitaries $U,V$ such that $UV=e^{2\pi i\theta }VU$ i.e. Rieffel deformation of $C(\IT^2)$ with respect to 
$\begin{pmatrix}
 0&\theta\\
-\theta&0 
 \end{pmatrix}
.$
The $C^*$ algebra underlying the quantum isometry group of the standard spectral triple on $C^*(\IT^2_\theta)$ (see \cite{connes}) is given by\\ $\oplus_{i=1}^4(C(\IT^2)\oplus C^*(\IT^2_\theta))$ (see \cite{jyoti-dg}). Let $U_{k1},U_{k2}$ be the generators of $C(\IT^2)$ for odd $k$ and $C^*(\IT^2_\theta)$ for even $k,~k=1,2,...8.$ Define 
$$M=
\left(
\begin{array}{llll}
A_1 & A_2 & C_1^* & C_2^*\\
B_1 & B_2 & D_1^* & D_2^*\\
C_1 & C_2 & A_1^* & A_2^*\\
D_1 & D_2 & B_1^* & B_2^*
\end{array}
\right)
,$$ where $A_1=U_{11}+U_{41},$ $A_2=U_{62}+U_{72},$ $B_1=U_{52}+U_{61},$ $B_2=U_{12}+U_{22},$
$C_1=U_{21}+U_{31},$ $C_2=U_{51}+U_{82},$ $D_1=U_{71}+U_{81},$ $D_2=U_{32}+U_{42}.$ Then the coproduct $\Delta$ and the counit $\epsilon$ are given by $\Delta(M_{ij})=\sum_{k=1}^4 M_{ik}\ot M_{kj},$ $\epsilon(M_{ij})=\delta_{ij}.$ The action of the QISO on $C^*(\IT^2_\theta),$ say $\alpha,$ is given by 
\begin{equation*}
\begin{split}
&\alpha(U)=U\ot(U_{11}+U_{41})+V\ot(U_{52}+U_{61})+U^{-1}\ot(U_{21}+U_{31})+V^{-1}\ot(U_{71}+U_{81}),\\
&\alpha(V)=U\ot(U_{62}+U_{72})+V\ot(U_{12}+U_{22})+U^{-1}\ot(U_{51}+U_{82})+V^{-1}\ot(U_{32}+U_{42}).
\end{split}
\end{equation*}

\item {\bf The $\theta$ deformed sphere $S^{2n-1}_\theta$:} The non-commutative manifold $S^{2n-1}_\theta,$ for a skew symmetric matrix $\theta$ is the universal $C^*$-algebra generated by $2n$ elements $\{z^\mu,\overline{z}^\mu\}_{\mu=1,2,..2n},$ satisfying the relations:
\begin{itemize}
\item $(z^\mu)^*=\overline{z}^\mu;$
\item $z^\mu z^\nu=e^{2\pi i\theta_{\mu\nu}}z^\nu z^\mu,$ $\overline{z}^\mu z^\nu=e^{2\pi i\theta_{\nu\mu}}z^\nu\overline{z}^\mu;$
\item $\sum_{\mu=1}^{2n}z^\mu\overline{z}^\mu=1.$ 
\end{itemize}

The quantum isometry group of the spectral triples on $S^n_\theta,$ as described in \cite{connes,connes-dubois} is $O_\theta(n)$ whose CQG structure is described as follows: It is generated by $(a^{\mu}_{\nu},b^{\mu}_{\nu})_{\mu,\nu=1,2,...n},$ satisfying:
\begin{enumerate}
\item $a^{\mu}_{\nu}a^{\tau}_{\rho}=\lambda_{\mu\tau}\lambda_{\rho\nu}a^{\tau}_{\rho}a^{\mu}_{\nu},$ 
$a^{\mu}_{\nu}a^{*\tau}_{\rho}=\lambda^{\tau\mu}\lambda_{\nu\rho}a^{*\tau}_{\rho}a^{\mu}_{\nu},$ \item $a^{\mu}_{\nu}b^{\tau}_{\rho}=\lambda_{\mu\tau}\lambda_{\rho\nu}b^{\tau}_{\rho}a^{\mu}_{\nu},$ $a^{\mu}_{\nu}b^{*\tau}_{\rho}=\lambda^{\tau\mu}\lambda_{\nu\rho}b^{*\tau}_{\rho}a^{\mu}_{\nu},$ \item $b^{\mu}_{\nu}b^{\tau}_{\rho}=\lambda_{\mu\tau}\lambda_{\rho\nu}b^{\tau}_{\rho}b^{\mu}_{\nu},$ $b^{\mu}_{\nu}b^{*\tau}_{\rho}=\lambda^{\tau\mu}\lambda_{\nu\rho}b^{*\tau}_{\rho}b^{\mu}_{\nu},$ \item $\sum_{\mu=1}^n(a^{*\mu}_{\alpha}a^{\mu}_{\beta}+b^{\mu}_{\alpha}b^{*\mu}_{\beta})=\delta_{\alpha\beta}1,$ $\sum_{\mu=1}^n(a^{*\mu}_{\alpha}b^{\mu}_{\beta}+b^{\mu}_{\alpha}a^{*\mu}_{\beta})=0,$ 
\end{enumerate}

The coproduct $\Delta$ is given by
$\Delta(a^{\mu}_{\nu})=\sum_{\lambda=1}^na^{\mu}_{\lambda}\ot a^{\lambda}_{\nu}+\sum_{\lambda=1}^nb^{\mu}_{\lambda}\ot b^{*\lambda}_{\nu},$\\ $\Delta(b^{\mu}_{\nu})=\sum_{\lambda=1}^na^{\mu}_{\lambda}\ot b^{\lambda}_{\nu}+\sum_{\lambda=1}^nb^{\mu}_{\lambda}\ot a^{*\lambda}_{\nu};$  
and the counit $\epsilon$ is given by $\epsilon(a^{\mu}_{\nu})=\delta_{\mu\nu},$ $\epsilon(b^{\mu}_{\nu})=0.$

The action of the QISO on $S^{2n-1}_\theta,$ say $\alpha$ is given by
\begin{equation*}
\alpha(z^\mu)=\sum_\nu(z^\nu\ot a^\mu_\nu +\overline{z}^\nu\ot b^\mu_\nu),~\alpha(\overline{z}^\mu)=\sum_{\nu}(\overline{z}^\nu\ot\overline{a}^\mu_\nu+z^\nu\ot\overline{b}^\mu_\nu).  
\end{equation*}
 
\item {\bf The free sphere $S^{n-1}_{+}$: } The free sphere denoted by $S^{n-1}_{+}$ is defined as the universal $C^*$ algebra generated by elements 
$\{x_{i}\}_{i=1}^{n-1}$ satisfying $x_{i}=x_{i}^*$ and $\sum_{i=1}^{n-1}x^2_i=1.$  Consider the spectral triples as described in Theorem 6.4 in page 13 of \cite{dg-banica}. It has been shown (see \cite{dg-banica}) that the quantum isometry group associated to this spectral triple is the free orthogonal group $C^*(O_+(n))$ which is described as the universal $C^*$-algebra generated by $n^2$ elements $\{x_{ij}\}_{i,j=1}^n$ satisfying 
\begin{enumerate}
\item[a.] $x_{ij}=x_{ij}^*$ for $i,j=1,2,...n;$
\item[b.] $\sum_{k=1}^nx_{ki}x_{kj}=\delta_{ij}{\bf 1},$ $\sum_{k=1}^nx_{ik}x_{jk}=\delta_{ij}{\bf 1}.$
\end{enumerate}
\end{enumerate}
For more examples, we refer the reader to \cite{jyoti}.
 
\subsubsection{Algebraic Theory of Levy processes on involutive bialgebras}
We refer the reader to \cite{franz1} and \cite{schurmann1} for the basics of the algebraic theory of Levy processes on involutive bialgebras, which we briefly review here.
\bdfn
Let $\clb$ be an involutive bialgebra with coproduct $\Delta$. A quantum stochastic process $(l_{st})_{0\leq s\leq t}$ on $\clb$ over some quantum probability space $(\cla,\Phi)$ (i.e. $\cla$ is a unital $\ast$-algebra, $\Phi$ is a positive functional such that $\Phi(1)=1$) is called a Levy process, if the following four conditions are satisfied:
\begin{enumerate}
\item (increment property) We have $l_{rs}*l_{st}=l_{rt}$ for all $0\leq r\leq s\leq t,$  $l_{tt}=1\circ\epsilon$ for all $t\geq0,$ where $l_{rs}*l_{st}:=m_{\cla}\circ(l_{rs}\ot l_{st})\circ\Delta.$
\item (independence of increments) The family $(l_{st})_{0\leq s\leq t}$ is independent, i.e. the quantum random variables $l_{s_1t_1},l_{s_2t_2},....l_{s_nt_n}$ are independent for all $n\in\IN$ and all $0\leq s_1\leq t_1\leq....t_n.$
\item (Stationarity of increments) The marginal distribution $\phi_{st}:=\Phi\circ l_{st}$ of $j_{st}$ depends only on the difference $t-s.$
\item (Weak continuity) The quantum random variables $l_{st}$ converge to $l_{ss}$ in distribution for $t\rightarrow s.$\end{enumerate}
\edfn
Define $l_t:=l_{0t}.$

Due to stationarity of increments, it is meaningful to define the marginal distributions of $(l_{st})_{0\leq s\leq t}$ by $\phi_{t-s}=\Phi\circ l_{st}.$
\begin{lmma}(\cite{franz1}).
The marginal distributions $(\phi_t)_{t\geq0}$ form a convolution semigroup of states on $\clb$ i.e. they satisfy 
\begin{enumerate}
\item $\phi_0=\epsilon,~\phi_t*\phi_s=\phi_{t+s}$ for all $s,t\geq0,$ and $\lim_{t\rightarrow0}\phi_t(b)=\epsilon(b)$ for all $b\in\clb.$
\item $\phi_t(1)=1$ and $\phi_t(b^*b)\geq0$ for all $t\geq0$ and all $b\in\clb.$
\end{enumerate}
\end{lmma}
This convolution semigroup characterizes a Levy process on an involutive bialgebra.
\bdfn
A functional $l:\clb\rightarrow\IC$ is called conditionally completely positive (CCP for short) functional if $l(b^*b)\geq0$ whenever $\epsilon(b)=0.$
\edfn 
The generator of the above convolution semigroup of states is a CCP functional on the bialgebra $\clb.$
\begin{ppsn}({\bf Schoenberg correspondence})\cite{franz1}
Let $\clb$ be an involutive bialgebra, $(\phi_t)_{t\geq0}$ a convolution semigroup of linear functionals on $\clb$ and $l$ be its generator, i.e. $l(a)=\frac{d}{dt}|_{_{t=0}}\phi_t(a).$. Then the following are equivalent:
\begin{enumerate}
 \item $(\phi_t)_{t\geq0}$ is a convolution semigroup of states.
\item $l:\clb\rightarrow\IC$ satisfies $l(1)=0,$ and it is hermitian and CCP.
\end{enumerate}
\end{ppsn}
Next we define Schurmann triple on $\clb.$
\bdfn
Let $\clb$ be a unital $\ast$-algebra equipped with a unital hermitian character $\epsilon.$ A Schurmann triple on $(\clb,\epsilon)$ is a triple $(\rho,\eta,l)$ consisting of 
\begin{enumerate}
 \item a unital $\ast$-representation $\rho:\clb\rightarrow\cll(\cld)$ of $\clb$ on some pre-Hilbert space $D,$
\item a $\rho-\epsilon-1$-cocycle $\eta:\clb\rightarrow \cld,$ i.e. a linear map $\eta:\clb\rightarrow D$ such that $$\eta(ab)=\rho(a)\eta(b)+\eta(a)\epsilon(b)$$ for all $a,b\in\clb,$
\item and a hermitian linear functional $l:\clb\rightarrow\IC$ that satisfies
$$l(ab)=l(a)\epsilon(b)+\epsilon(a)l(b)+\innerl\eta(a^*),\eta(b)\innerr$$ for all $b\in\clb.$
\end{enumerate}
\edfn
A Schurmann triple is called surjective if the cocycle $\eta$ is surjective. Upto unitary equivalence, we have a one-to-one correspondence between Levy processes on $\clb,$ convolution semigroup of states on $\clb$ and surjective Schurmann triples on $\clb.$ Choosing an orthonormal basis $(e_i)_i$ of $\cld,$ we can write $\eta$ as $\eta(\cdot)=\sum_i\eta_i(\cdot)e_i.$ The $\eta_i's$ will be called the `coordinate' of the cocycle $\eta.$

We will denote by $\clv_{_{\cla}},$ the vector space of $\epsilon$-derivations on $\cla_0,$ i.e. for $\clv_{_{\cla}}$ consists of all maps $\eta:\cla_0\rightarrow\IC,$ such that $\eta(ab)=\eta(a)\epsilon(b)+\epsilon(a)\eta(b).$
\begin{lmma}\label{bounded by the noise space}
Let $l$ be the generator of a Gaussian process on $\cla_0.$ Suppose that $(l.\eta.\epsilon)$ be the surjective Schurmann triple associated to $l.$ Let $d:=dim\clv_{_{\cla}}.$ Then there can be atmost $d$ coordinates of $\eta.$
\end{lmma}
\begin{proof}
Let $(\eta_i)_i$ be the coordinates of $\eta.$ Observe that $\eta_i$ is an $\epsilon$-derivation for all $i.$ It is enough to prove that $\{\eta_i\}_i$ is a linearly independent set. Suppose that 
$\sum_{i=1}^{k}\lambda_i\eta_i(a)=0,$ for all $a\in\cla_0.$ This implies that 
$\innerl\eta(a),\sum_{i=1}^{k}\lambda_ie_i\innerr=0,$ for all $a\in\cla_0,$ where $(e_i)_i$ is an orthonormal basis for $k_0,$ the associated noise space. Since $\{\eta(a):a\in\cla_0\}$ is total in $k_0,$ we have $\sum_{i=1}^{k}\lambda_ie_i=0$ which implies that $\lambda_i=0$ for $i=1,2,...k.$ Hence proved.
\end{proof}
\paragraph{}
\begin{ppsn}\cite{franz1}\label{franz}
For a ganerator $l$ of a Levy process, the following are equivalent:
\begin{enumerate}
\item $l|_{K^3}=0,$ $K=ker\epsilon,$
\item $l(b^*b)=0$ for all $b\in K^2,$
\item $l(abc)=l(ab)\epsilon(c)-\epsilon(abl(c)+l(bc)\epsilon(a)-\epsilon(bc)l(a)+l(ac)\epsilon(b)-\epsilon(ac)l(b),$
\item $\rho|_{K}=0,$ for any surjective Schurmann triple,
\item $\rho=\epsilon1$ for any surjective Schurmann triple i.e. the process is "Gaussian",
\item $\eta|_{K^2}=0$ for any Schurmann triple,
\item $\eta(ab)=\eta(a)\epsilon(b)+\epsilon(a)\eta(b)$ for any Schurmann triple.
\end{enumerate}
\end{ppsn}
A generator $l$ satisfying any of the above conditions is called a {\it Gaussian generator} or the generator of a Gaussian process.
\bdfn
We will call a Gaussian Levy process the algebraic Quantum Brownian Motion (QBM for short) if span of the maps $\{\eta_i\}_i$ is the whole of $\clv_{_{\cla}},$ where $\eta_i$ are the `coordinates' of the cocycle of the unique (upto unitary equivalnece) surjective Schurmann triple associated to $l.$
\edfn
\paragraph{}
It is known \cite{schurmann1} that the following weak  stochastic  equation
\begin{equation}\label{schurmann's equation}
\begin{split}
\innerl l_{t}(x)e(f),e(g)\innerr&=\epsilon(x)\innerl e(f),e(g)\innerr\\
&+\int_0^t d\tau\innerl\{l_{\tau}\ast(l+\innerl g(\tau),\eta\innerr+\innerl\eta,f(\tau)\innerr+\innerl g(\tau),(\rho-\epsilon)_{_{f(\tau)}}\innerr)\}(x)e(f),e(g)\innerr,
\end{split}
\end{equation}
which can be symbolically written as
\begin{equation*}
dl_{t}=l_{t}*(dA^\dagger_t\circ\eta+d\Lambda_t\circ(\rho-\epsilon)+dA_t\circ\eta^\dagger+l dt)
\end{equation*}
with the initial conditions $$l_{0}=\epsilon1$$ has a unique solution $(l_{st})_{0\leq s\leq t}$ such that $l_{st}$ is an algebraic levy process on $\cla_0.$ Then using this algebraic quantum stochastic differential equation, it can be proved that $j_t=\widetilde{l}_t$ satisfies an EH type equation as defined in subsection \ref{Quantum stochastics} with $\delta=\wtil{\eta},~\cll=\widetilde{l},~\sigma=\wtil{\rho-\epsilon}.$ However, it is not clear whether $j_t(x)\in\cla_0\ota B(\Gamma(L^2(\IR_+,k_0))).$ We shall prove later that at least for Gaussian generators, this will be the case i.e. $j_t(x)$ is bounded.
 
\section{Quantum Brownian Motion on non-commutative manifolds}
\subsection{Analytic construction of Quantum Brownian motion}\label{analytic}
Let $(\clq,\Delta)$ be a CQG, $\clq_0$ be the corresponding Hopf-$\ast$ algebra and $h$ be the Haar state on $\clq.$ Let $\clq_0:=\oplus\clh_\pi$ be the decomposition obtained by Peter-Weyl theory as in section \ref{CQG}.
\bthm\label{theorem_semi}
Let $(T_t)_{t\geq0}$ be a QDS on $\clq$ such that it is left covariant in the sense that\\ $(id\ot T_t)\circ\Delta=\Delta\circ T_t.$ Let $\cll$ be the generator of $(T_t)_{t\geq0}.$ Then there exist a CCP functional $l$ on $\clq_0$ such that $\wtil{l}=\cll.$
\ethm
\begin{proof}
The generator $\till$ is CCP in the sense that $\partial\cll(x,y)=\cll(x^*y)-\cll(x^*)y-x^*\cll{(y)}$ is a CP kernel (see \cite{dgkbs}). The left covariance condition implies that for each $t\geq0,$ $T_t$ as well as $\cll$ keep each of the spaces $\clh_\pi$ invariant. Consequently $\cll(\clq_0)\subseteq\clq_0,$ so that it makes sense to define $l=\epsilon\circ\cll.$ Moreover 
for $x,y\in\clq_0,$ $\epsilon\circ\partial\cll(x,y)=l\left((x-\epsilon(x))(y-\epsilon(y))\right),$ so that $l$ is CCP in schurmann's sense. Hence our claim is proved.
\end{proof}
We shall prove the converse of Theorem \ref{theorem_semi} for the Gaussian generators. For this, we need a few preparatory lemmae.
\begin{lmma}\label{kappa_elo}
In Sweedler's notation, $h(a_{(1)}b)a_{(2)}=h(ab_{(1)})\kappa(b_{(2)}).$ 
\end{lmma}
\begin{proof}
\begin{equation}
\begin{split}
 h(ab_{(1)})\kappa(b_{(2)})&=\left((h\ot1)\circ\Delta\right)(ab_{(1)})\kappa(b_{(2)})\\
&=(h\ot id)\left(\Delta(ab_{(1)})(id\ot\kappa)(b_{(2)})\right)\\
&=(h\ot id)\{\Delta(a)\Delta(b_{(1)})(id\ot\kappa(b_{(2)}))\}\\
&=(h\ot id)\left[\Delta(a)\{(id\ot m_{\clq})(\Delta\ot id)(id\ot\kappa)\Delta(b)\}\right]\\
&=(h\ot id)\left[\Delta(a)\{(id\ot m_{\clq})(id\ot id\ot\kappa)(\Delta\ot id)\Delta(b)\}\right]\\
&=(h\ot id)\left[\Delta(a)\{(id\ot m_{\clq})(id\ot id\ot\kappa)(id\ot\Delta)\Delta(b)\}\right]\\
&=(h\ot id)\left[\Delta(a)\{(id\ot m_{\clq}\circ(id\ot\kappa)\Delta)\Delta(b)\}\right]\\
&=(h\ot id)\left[\Delta(a)\{(id\ot\epsilon)\Delta(b)\}\right]=(h\ot id)\left[\Delta(a)(b\ot1)\right]\\
&=h(a_{(1)}b)a_{(2)}.
\end{split}
\end{equation}
\end{proof}
\begin{crlre}\label{adjoint_elo}
For any functional $\clp:\clq_0\rightarrow\IC,$ $h\left(\widetilde{\clp}(a)b\right)=h\left(a(\widetilde{\clp\circ\kappa})(b)\right).$
\end{crlre}
\begin{proof}
\begin{equation}
\begin{split}
h(\widetilde{\clp}(a)b)&=(h\ot id)\left[(id\ot\clp)\Delta(a)(b\ot1)\right]\\
&=(id\ot\clp)\left[(h\ot id)(\Delta(a)(b\ot1))\right]\\
&=(id\ot\clp)\left[h(ab_{(1)})\kappa(b_{(2)})\right]\\
&=h(ab_{(1)})\clp(\kappa(b_{(2)}))=h(ab_{(1)}\clp(\kappa(b_{(2)})))\\
&=h(a(id\ot\clp)(id\ot\kappa)\Delta(b))=h(a(\widetilde{\clp\circ\kappa})(b)).
\end{split}
\end{equation}
\end{proof}

\begin{lmma}\label{derivation_0}
Let $\eta:\clq_0\rightarrow\IC$ be an $\epsilon$-derivation. Put $\delta:=(id\ot\eta)\circ\Delta.$ Then $h(\delta(a))=0$ for all $a\in\clq_0.$
\end{lmma}
\begin{proof}
\begin{equation}
\begin{split}
h(\delta(a))&=(h\ot id)(id\ot\eta)\circ\Delta(a)\\
&=(id\ot\eta)(h\ot id)\circ\Delta(a)\\
&=\eta(h(a)1_{\clq})\\
&=h(a)\eta(1_{\clq})=0~\mbox{for all $a\in\clq_0,$}
\end{split}
\end{equation}
where we have used the fact that $(h\ot id)\circ\Delta(a)=(id\ot h)\circ\Delta(a)=h(a)1_{\clq}.$
\end{proof}
Let $(l,\eta,\epsilon)$ be the surjective Schurmann triple for $l,$ so that on $\clq_0,$ we have\\ 
$l(a^*b)-\epsilon(a^*)l(b)-l(a^*)\epsilon(b)=\innerl\eta(a),\eta(b)\innerr.$ We recall that $\eta:\clq_0\rightarrow k_0,$ for some Hilbert space $k_0$ so that $\eta(a)=\sum_i\eta_i(a)e_i,$ $(e_i)_i$ being an orthonormal basis for $k_0$ and $\eta_i:\clq_0\rightarrow\IC$ being an $\epsilon$-derivation for each $i.$ Define 
$\theta^i_0:=(id\ot\eta_i)\circ\Delta$ for each $i.$ Observe that\\ $\|\sum_i\theta^i_0(x)^*\theta^i_0(x)\|\leq\|x_{(1)}^*x_{(1)}\|~|\sum_i\overline{\eta_i(x_{(2)})}\eta_i(x_{(2)})|<\|x_{(1)}\|^2\|\eta(x_{(2)})\|^2<\infty,$ so that $\delta:=\sum_i\theta^i_0\ot e_i=(id\ot\eta)\circ\Delta$ is a derivation from $\clq_0$ to $\clq\ot k_0.$ Now $\cll$ is a densely defined operator with $D(\cll)=\clq_0\subseteq L^2(h).$ By Corollary \ref{adjoint_elo}, $h(\cll(a^*)b)=h(a^*\widetilde{l\circ\kappa}(b))$ i.e. $\innerl\cll(a),b\innerr_{L^2(h)}=\innerl a,\widetilde{l\circ\kappa}(b)\innerr_{L^2(h)}.$ Thus $\cll$ has an adjoint which is also densely defined. Thus $\cll$ is $L^2(h)$-closable, and we denote its closure by the same notation 
$\cll.$ Note that a linear map $S:\clq_0\rightarrow\clq_0$ is left covariant i.e. 
$(id\ot S)\Delta=\Delta\circ S$ if and only if $S(\clh_\pi)\subseteq\clh_\pi$ for all $\pi.$ In such a case, we will denote by $S_\pi$ the map $S|_{_{\clh_\pi}}.$

\begin{lmma}\label{christinsen}
Let $l:\clq_0\rightarrow\IC$ be a CCP functional and $(l,\eta,\epsilon)$ be the surjective Schurmann triple associated with it. Then $\cll=\widetilde{l}$ on $\clq_0$ has Christinsen-Evans form i.e. 
$$\cll(x)=R^*(x\ot1_{k_0})R-\frac{1}{2}R^*Rx-x\frac{1}{2}R^*R+i[T,x],$$ for densely defined closable operators $R$ and $T,$ $T^*=T.$
\end{lmma}
\begin{proof}
Let $R:=\delta:\clq_0\left(\subseteq L^2(h)\right)\rightarrow L^2(h)\ot k_0,$ where $\delta:=(id\ot\eta)\circ\Delta.$\\ For $x\in\clq_0,$ consider the quadratic forms  
\begin{equation}\label{1}
\innerl\Phi(x)y,y^\prime\innerr_{L^2(h)}=h\left(\cll(y^*x^*y^\prime)-\cll(y^*x^*)y^\prime-y^*\cll(x^*y^\prime)+y^*\cll(x^*)y^\prime\right);
\end{equation}
 \begin{equation}\label{2}
\innerl\cll(x)y,y^\prime\innerr_{L^2(h)}=h(y^*\cll(x^*)y^\prime)
\end{equation}
and 
\begin{equation}\label{3}
\frac{1}{2}\innerl[\cll-\widetilde{l\circ\kappa},x]y,y^\prime\innerr_{L^2(h)}=\frac{1}{2}h(\cll(y^*x^*)y^\prime-y^*x^*\cll(y^\prime)-\cll(y^*)x^*y^\prime+y^*\cll(x^*y^\prime)),
\end{equation}
where $\Phi(x)=R^*(x\ot1_{k_0})R-\frac{1}{2}R^*Rx-x\frac{1}{2}R^*R.$ 
Observe that by subtracting (\ref{2}) from (\ref{1}) and adding (\ref{3}) to it, we get zero. So by taking\\ 
$T=\frac{1}{2i}(\cll-\widetilde{l\circ\kappa})$ on $\clq_0,$ we get $\innerl\cll(x)y,y^\prime\inner=\innerl\left(\Phi(x)+i[T,x]\right)y,y^\prime\inner$ for $x\in\clq_0.$ Note that $T$ is covariant, hence we have $T=\oplus_{\pi}T_\pi$ and since each $\clh_\pi$ is finite dimensional and $T_\pi^*=T_\pi$ by corollary \ref{adjoint_elo}, we have that $T$ has a self-adjoint extension on $L^2(h)$ which is the $L^2$-closure of $T$ in $\clq_0.$
\end{proof}
For a set of vectors $\{h_1,h_2,....\}$ in any vector space, we will denote by $\innerl h_i|i=1,2,...\innerr_\IC$ algebraic linear span over $\IC.$ We are now in a position to prove the converse of Theorem \ref{heat-semigroup} for Gaussian generators, which gives a left covariant QDS on $\clq$ and Gaussian generator.

\bthm\label{heat-semigroup}
Given a Gaussian CCP functional $l$ on $\clq_0,$ there is a unique covariant QDS on $\clq$ such that its generator is an extension of $\widetilde{l}.$
\ethm
\begin{proof}
Note that in notation of Lemma \ref{christinsen}, we have $R^*R=\widetilde{l}+\widetilde{l\circ\kappa},$ $T=\frac{1}{2i}(\widetilde{l}-\widetilde{l\circ\kappa})$ and hence $G:=iT-\frac{1}{2}R^*R=-\wtil{l\circ\kappa}.$ Hence $(id\ot G)\circ\Delta=\Delta\circ G.$ So each $G_\pi$ generates a semigroup in $\clh_\pi$ say $T_t^\pi$ which is contractive, since the generator is of the form $iT_\pi-\frac{1}{2}(R^*R)_\pi,$ with $T_\pi^*=T_\pi.$ Take $S_t:=\oplus_\pi T_t^\pi,$ which is a $C_0,$ contractive semigroup in $L^2(h).$ There exists a minimal semigroup $(T_t)_{t\geq0}$ on $B(L^2(h)),$ such that its generator, say $\cll^{min},$ is of the form given in Lemma \ref{christinsen} when restricted to a suitable dense domain (see \cite{dgkbs}). Now following the arguments used in proving Theorem \ref{THE THEOREM}, we can conclude that $\cll^{min}=\widetilde{l}$ on $\clq_0.$ Thus $\cll^{min}(\clh_\pi)\subseteq\clh_\pi.$ Furthermore, each $\clh_\pi$ being finite dimensional, $T_t(x)=e^{t\cll_\pi^{min}}(x)=\sum_n\frac{t^n}{n!}(\cll^{min}_\pi)^n(x),$ which converges in the norm for $x\in\clh_\pi.$ Thus in particular we see that $T_t(\clh_\pi)\subseteq\clh_\pi$ for all $\pi$ and all $t\geq0$ i.e. $(id\ot T_t)\circ\Delta=\Delta\circ T_t.$ 
\end{proof}

\bthm\label{EH elo}
The QDS generated by a Gaussian generator $l$ as in Theorem \ref{heat-semigroup}, always admits E-H dilation which is implemented by unitary cocycles.
\ethm
\begin{proof} We will apply Theorem \ref{THE THEOREM} with $H=T.$ Let $\clv_0=\clq_0$ and $\clw_0=\langle e_i|i=1,2,3...\rangle_\IC,$ where $(e_i)_i$ is an orthonormal basis for $k_0.$ Observe that by Lemma \ref{derivation_0},  $R^*=-\sum_i\theta^0_i\ot\langle e_i|.$ Thus $u\ot\xi\in D(R^*)$ for all $u\in\clv_0$ and $\xi\in\clw_0.$ The proof of Theorem \ref{heat-semigroup} implies that $G:=iT-\frac{1}{2}R^*R$ generates a $C_0$ contractive semigroup in $L^2(h).$ Noting that $G^*$ is an extension of $-\widetilde{l},$ using arguments as in Theorem \ref{heat-semigroup}, we can prove that $G^*$ generates a $C_0$ contractive semigroup in $L^2(h).$ Thus all the conditions of Theorem \ref{THE THEOREM} hold, and we get unitary cocycles $(U_t)_{t\geq0}$ satisfying an H-P equation. Then $j_t:B(L^2(h))\rightarrow B(L^2(h))\ot B(\Gamma)$ defined by $j_t(x):=U_t(x\ot1_\Gamma)U_t^*,$ is a $\ast$-homomorphic EH flow satisfying the stochastic differential equation:
\begin{equation}\label{EH}
\begin{split}
dj_t=&j_t\circ\left(a_{\delta^\dagger}(dt)-a^\dagger_\delta(dt)+\cll dt\right)\\
&j_0=id, 
\end{split}
\end{equation}
on $\clq_0,$ where $\delta(x)=(id\ot\eta)\circ\Delta(x)=[R,x]$ for $x\in\clq_0.$ We need to show that $j_t(\clq_0)\subseteq\clq^{\prime\prime}\ot B(\Gamma)$ i.e. $\innerl e(f),j_t(x)e(g)\innerr\in\clq^{\prime\prime}$ for $f,g\in\Gamma$ and $x\in\clq_0.$

Let $l_t$ be the algebraic Levy process associated with $l,$ satisfying equation (\ref{schurmann's equation}) with $\rho=\epsilon.$ 
For $x\in\clq_0$ and $\xi,\xi^\prime$ belonging to $k_0,$ let 
$T^{\xi,\xi^\prime}_t$ and 
$\phi^{\xi,\xi^\prime}_t$ denote the maps
$\innerl e(\chi_{[0,t]}\xi),j_t(\cdot)_{e(\chi_{[0,t]}\xi^\prime)}\innerr$ and 
$\innerl e(\chi_{[0,t]}\xi),l_t(\cdot){e(\chi_{[0,t]}\xi^\prime)}\innerr$
repectively. We claim that for all $x\in\clq_0,$ $T^{\xi,\xi^\prime}_t(x)=\wtil{\phi^{\xi,\xi^\prime}_t}(x)$ which will be shown towards the end of the proof. Let $\cld$ denote the linear span of elements of the form $e(f)$ where $f$ is a step function taking values in $(e_i)_i.$ By the theorems in \cite{skeide,krpsunder}, $\cld$ is dense in $\Gamma.$ Consider the step functions $f=\sum_i^k a_i\chi_{[t_{i-1},t_i]}$ and $g=\sum_i^kb_i\chi_{[t_{i-1},t_i]},$ where $t_0=0,t_k=t,$ and $a_i,b_i$ belong to $\{e_i:i\in\IN\}.$ Then note that for $x\in\clq_0,$
\begin{equation}
\begin{split}
\innerl e(f),j_t(x)_{e(g)}\innerr&=T_{t_1-t_0}^{a_1,b_1}\circ T_{t_2-t_1}^{a_2,b_2}\circ......T^{a_k,b_k}_{t-t_{k-1}}(x)\\
&=\wtil{\phi_{t_1-t_0}^{a_1,b_1}}\circ\wtil{\phi_{t_2-t_1}^{a_2,b_2}}\circ......\wtil{\phi^{a_k,b_k}_{t-t_{k-1}}}(x)\\
&=\wtil{A}(x)\\
&=\innerl e(f),\wtil{l}_t(x)e(g)\innerr~\in\clq_0,
\end{split}
\end{equation}

where $A(x)=(\phi_{t_1-t_0}^{a_1,b_1}\ast \phi_{t_2-t_1}^{a_2,b_2}\ast......\phi^{a_k,b_k}_{t-t_{k-1}})(x).$  Since $\cld$ is total in $\Gamma,$ this implies that the map
$\innerl e(f),j_t(x)e(g)\innerr\in\clq^{\prime\prime}$ for all $f,g\in\Gamma~x\in\clq_0.$ 

The proof of the theorem will be complete once we show that for $x\in\clq_0,$ $\xi,\nu\in k_0,$ we have  $T^{\xi,\nu}_t(x)=\wtil{\phi^{\xi,\nu}_t}(x).$ This can be achieved as follows:

Fix an $x\in\clq_0.$ From the cocycle property, it follows that $T^{\xi\nu}_t$ is a $C_0$-semigroup on $\clq$ and $\phi_t^{\xi,\nu}$ is a convolution semigroup of states on $\clq_0.$ Since $l_t$ and $j_t$ satisfy equations (\ref{schurmann's equation}) and (\ref{EH}) respectively, it follows that the generator of the convolution semigroup $(\phi^{\xi,\nu}_t)_{t\geq0}$ is $L=l+\innerl\xi,\eta\innerr+\eta^\dagger_\nu$ and the generator of the semigroup $(T_t^{\xi\nu})_t$ is $\widetilde{L}.$ By the fundamental Theorem of coalgebra (see \cite{franz1}), there is a finite dimensional coalgebra say $C_x$ containing $x.$ It follows that $\wtil{L}(C_x)\subseteq C_x.$ Note that $C_x$ being finite dimensional, the map 
$\wtil{L}:C_x\rightarrow C_x$ is bounded with $\|\widetilde{L}\|=M_x$(say), where $M_x$ depends on $x.$ Now 
\begin{equation}
\begin{split}
T^{\xi,\nu}_t(x)&=x+\int_0^tT_s^{\xi,\nu}(\wtil{L}(x))ds\\
&=x+t\wtil{L}(x)+\int_{s_1=0}^t\int_{s_2=0}^{s_1}T_{s_2}^{\xi,\nu}(\wtil{L}^2(x))ds\\
&=x+t\wtil{L}(x)+\frac{t^2}{2!}\wtil{L}^2(x)+\frac{t^3}{3!}\wtil{L}^3(x)+.....
+\int_{s_1=0}^t\int_{s_2=0}^{s_1}\int_{s_3=0}^{s_2}....\int_{s_n=0}^{s_{n-1}}T^{\xi,\nu}_{s_n}(\wtil{L}^n(x))ds;
\end{split}
\end{equation}
Now $\|\int_{s_n=0}^{s_{n-1}}T^{\xi,\nu}_{s_n}(\wtil{L}^n(x))ds\|\leq e^{t\innerl\xi,\nu\innerr}\frac{t^n(M_x)^n\|x\|}{n!}\rightarrow 0$ as $n\rightarrow\infty.$
Thus 
\begin{equation}
\begin{split}
T_t^{\xi,\nu}(x)&=x+t\wtil{L}(x)+\frac{t^2}{2!}\wtil{L}^2(x)+......\\
&=\wtil{\epsilon}(x)+t\wtil{L}(x)+\frac{t^2}{2!}\wtil{(L\ast L)}(x)+\frac{t^3}{3!}\wtil{(L\ast L\ast L)}(x)+....\\
&=\wtil{\phi^{\xi,\nu}_t}(x),~where\\
&=\phi_t^{\xi,\nu}(x)=\left(\epsilon+tL+\frac{t^2}{2!}(L\ast L)+\frac{t^3}{3!}(L\ast L\ast L)+....\right)(x).
\end{split}
\end{equation}
\end{proof}

We will call $j_t$ a Quantum Gaussian process on $\clq.$ If $l$ generates the algebraic QBM (as defined after Proposition \ref{franz}), then we will call $j_t$ the Quantum Brownian motion (QBM for short) on $\clq.$
\begin{rmrk}
If $l=l\circ\kappa,$ we will call the above QBM symmetric. This is because under the given condition, $(T_t)_{t\geq0}$ generated by $\cll$ becomes a symmetric QDS i.e. $h(T_t(x)y)=h(xT_t(y)).$
\end{rmrk}
The following result, which is probably well-known, demonstrates the equivalence of the quantum and classical definitions of Gaussian processes on compact Lie-groups. 
\bthm
Let $G$ be a compact Lie-group. Then a generator of a quantum Gaussian process (QBM) on $\clq=C(G)$ is also the generator of a classical Gaussian process (QBM) and vice-versa.
\ethm
\begin{proof}
Let $l$ be the given generator and let $\cll:=\widetilde{l},$ as before. Observe that the semigroup $(T_t)_{t\geq0}$ associated with the map $\cll$ is covariant with respect to left action of the group. Moreover, $(T_t)_{t\geq0}$ is a Feller semigroup. Thus by Theorem $2.1$ in page $42$ of \cite{mingliao}, we see that $C^\infty(G)\subseteq D(\cll).$ Now on $C(G),$ there is a canonical locally convex topology generated by the seminorms $\|f\|_n:=\sum_{i_1,i_2,...i_k:k\leq n}\|\partial_{i_1}\partial_{i_2}...\partial_{i_k}(f)\|,$ where $\partial_{i_l}$ is the generator of the one-parameter group $L_{exp(tX_{i_l})},$ such that $C^\infty(G)$ is complete and $\clq_0$ is dense in $C^\infty(G)$ in this topology (see \cite{dgkbs} and references therein).  Now as $\cll$ is closable in the norm topology, it is closable in this locally convex topology and hence (by the closed graph theorem) continuous as a map from $\left(C^\infty(G),\{\|\cdot\|_n\}_n\right)\rightarrow (C(G),\|\cdot\|_\infty).$ From this, and using the fact that $\cll$ commutes with $L_g~\forall g,$ it can be shown along the lines of Lemma 8.1.9 in page 193 of \cite{dgkbs} that $\cll(f)\in C^\infty(G)$. Moreover, we can extend the identity 
$\cll(abc)=\cll(ab)c-ab\cll(c)+a\cll(bc)-\cll(a)bc+\cll(ac)b-ac\cll(b)$ for all $a,b,c\in C^\infty(G)$ by continuity. Thus $\cll$ is a local operator. Now by the main theorem in \cite{yosida}, this implies that $\cll$ is a second order elliptic differential operator, and hence generator of a classical Gaussian process. 

On the other hand, given a generator $\cll$ of a classical Gaussian process, $(id\ot\cll)\Delta=\Delta\circ\cll$ implies that in particular, $\cll(\clq_0)\subseteq\clq_0.$ Moreover, it can be verified that $\cll$ satisfies the identity $\cll(abc)=\cll(ab)c-ab\cll(c)+a\cll(bc)-\cll(a)bc+\cll(ac)b-ac\cll(b)$ for $a,b,c\in C^\infty(G)$ and hence in $\clq_0.$ Thus $\cll$ is the generator of a quantum Gaussian process as well. 
\end{proof}

\subsection{Quantum Brownian motion on quantum spaces}\label{notun section}
Suppose $G$ is a compact Lie-group, with Lie-algebra $\mathfrak{g},$ of dimension $d.$ There exists an $Ad(G)$-invariant inner product in $\mathfrak{g}$ which induces a bi-invariant Riemannian metric in $G.$ Suppose $G$ acts transitively on a manifold $M.$ Then as manifolds, $M\cong H\backslash G,$ for some closed subgroup $H\subseteq G$ and as the innerproduct on $\mathfrak{g}$ is in particular $Ad(H)$-invariant, it induces a $G$-invariant Riemannian metric on $M.$  Let $\mathfrak{g}$ and $\mathfrak{h}$ be the Lie-algebras of $G$ and $H$ respectively. It is a well-known fact (see \cite{mingliao}) that $\mathfrak{g}=\mathfrak{h}\oplus\mathfrak{p},$ where $\mathfrak{p}$ is a subspace such that $Ad(H)\mathfrak{p}\subseteq\mathfrak{p}$ and $[\mathfrak{p},\mathfrak{p}]\subseteq\mathfrak{h}.$ Let $\{\clx_i\}_{i=1}^d$ be a basis of $\mathfrak{g}$ such that $\{\clx_i\}_{i=1}^m$ is a basis for $\mathfrak{p}$ and $\{\clx_i\}_{i=m+1}^d$ is a basis for $\mathfrak{h}.$ Let $\pi:G\rightarrow H/G$ be the quotient map given by $\pi(g)=Hg,$ for $g\in G.$ It follows that if $f\in C(H/G),$ $\clx_i(f\circ\pi)\equiv0$ for all $i=m+1,...d.$  The Laplace-Beltrami operator on $M$ is thus given by \begin{equation*}
\frac{1}{2}\Delta_{H/G}f(x)=\frac{1}{2}\sum_{i=1}^{m}\clx_i^2(f\circ\pi)(g),
\end{equation*}
where $f\in C^\infty(M)$ and $x=Hg,$ or in other words, if $\{W_t^{(i)}\}_{i=1}^d$ denote the standard Brownian motion in $\IR^d,$ the standard covariant Brownian motion on $M(\cong H/G),$ starting at $m$ is given by $\clb_t^m:=m.\clb_t$ where $\clb_t:=exp(\sum_{i=1}^mW_t^{(i)}\clx_i)$ and $exp$ denotes the exponential map of the Lie group $G$.
Now suppose $M$ is a compact Riemannian manifold such that the isometry group of $M,$ say $G,$ acts transitively on $M.$ The above discussion applies to $M$ and it may be noted in particular that in this case, the Laplace-Beltrami operator on $M$ coincides with the Hodge-Laplacian on $M$ restricted to $C^\infty(M).$ It follows from Proposition $2.8$ in page $51$ of \cite{mingliao} and the discussions preceeding it that a Riemannian Brownian motion on a compact Riemannian manifold $M$ is induced by a bi-invariant Brownian motion on $G,$ the isometry group of $M,$ if $G$ acts transitively on $M.$ Furthermore by Proposition \ref{podles}, it follows that if $G$ acts transitively on $M,$ then the action is ergodic i.e. $C(M)$ is homogeneous.  Motivated by this, we amy define a Quantum Brownian motion on a quantum space as follows:

Let $(\cla^\infty,\clh,D)$ be a spectral triple satisfying the conditions stated in subsection \ref{QISO}. Let $(\clq,\Delta)$ denote the quantum isometry group as obtained in Proposition \ref{DG}, $\alpha$ being the action. Suppose that $(\clq,\Delta)$ acts ergodically on $\cla^\infty,$ {\bf i.e. the quantum space $\cla:=\overline{\cla^\infty}^{\|\cdot\|_\infty}$ is homogeneous}. Let $l:\clq_0\rightarrow\IC$ be the generator of a bi-invariant quantum gaussian process $j_t(\cdot)$ on $\clq$ i.e. $(l\ot id)\circ\Delta=(id\ot l)\circ\Delta$ on $\clq_0.$ 

Define the process $k_t:=(id\ot l_t)\circ\Delta:\cla_0\rightarrow\cla\ot B(\Gamma(L^2(\IR_+,k_0)))$ on $\cla_0.$ 
Since $\alpha$ is an ergodic action, it is known that there exists an $\alpha$-invariant state $\tau$ on $\cla$ (see \cite{}). Moreover, in the notation of Proposition \ref{spectral subspaces of action}, $\tau$ is faithful on 
$\cla_0:=\oplus_{\gamma\in Irr_{_{\clq}}}\left(\oplus_{i\in I_{\gamma}}W_{_{\gamma i}}\right)$ and as a Hilbert space, $L^2(\tau):=\oplus_{\gamma\in Irr_{_{\clq}}}\left(\oplus_{i\in I_{\gamma}}W_{_{\gamma i}}\right).$
\bthm\label{QBM on quantum spaces}
There exists a unitary cocycle $(U_t)_{t\geq0}\in Lin(L^2(\tau)\ot\Gamma)$ satisfying an HP equation, where $\Gamma:=\Gamma(L^2(\IR_+,k_0))$ such that $k_t(x)=U_t(x\ot id_{\Gamma})U_t^*$ for $x\in\cla_0.$ Thus $k_t$ extends to a bounded map from $\cla$ to $\cla^{\prime\prime}\ot B(\Gamma).$ Moreover, $k_t$ satisfies an EH equation with coefficients $(\cll_{_{\cla}},\delta,\delta^\dagger),$ where $\cll_\cla:=(id\ot l)\circ\alpha,$ $\delta:=(id\ot \eta)\circ\alpha,$ and initial condition $j_0=id.$
\ethm

\begin{proof}
Observe that $(\alpha\ot id)\circ\alpha=(id\ot\Delta)\circ\alpha.$ Hence, proceeding as in subsection \ref{analytic}, with $H_\pi$ replaced by $W_{_{\gamma}}$ for $\gamma\in Irr_{_{\clq}},$ and $L^2(h)$ replaced by $L^2(\tau),$ we get the existence of a unitary cocycle $(U_t)_{t\geq0}$ satisfying an HP equation with coefficient matrix
$\begin{pmatrix}
iT-\frac{1}{2}R^*R&R^*\\
R&0
 \end{pmatrix},$
with the initial condition $U_0=I,$
where $T,R$ are the closed extensions of $\frac{1}{2i}(\cll_{_{\cla}}-(id\ot (l\circ\kappa))\circ\alpha)$ and $(id\ot\eta)\circ\alpha$ respectively. Now, proceeding as in Theorem \ref{EH elo}, we get our result.
\end{proof}

\bdfn
A generator of covariant quantum Gaussian process (QBM) on the non-commutative manifold $\cla$ is defined as a map of the form $l_{_{\cla}}:=(id\ot l)\circ\alpha,$ where $l$ is the generator of some bi-invariant quantum Gaussian process (QBM) on $\clq.$ In such a case, the EH flow $k_t$ obtained in Theorem \ref{QBM on quantum spaces} will be called covariant quantum-Gaussian process (QBM) with the generator $\cll_\cla.$
\edfn
We will usually drop the adjective `covariant'.
Observe that 
\begin{equation}\label{covariance condition}
\begin{split}
(\cll_{_{\cla}}\ot id_\clq)\alpha=&(id_\cla\ot l\ot id_\clq)(\alpha\ot id_\clq)\alpha\\
&=(id_\cla\ot l\ot id_\clq)(id_\cla\ot\Delta)\alpha\\
&=(id_\cla\ot(l\ot id_\clq)\Delta)\alpha\\
&=(id_\cla\ot(id_\clq\ot l)\Delta)\alpha~\mbox{(since $(l\ot id)\Delta=(id\ot l)\Delta$)}\\
&=(\alpha\ot l)\alpha=\alpha\circ \cll_{_{\cla}}.
\end{split}
\end{equation}

It is not clear whether the condition (\ref{covariance condition}) is equivalent to the bi-invariance of the Gaussian generator $l$ on $\clq.$ However, let us show that it is indeed so for the class of quantum spaces which are quotient (hence in particular for the classical ones).

We recall (see subsection \ref{3-spaces}) that $\cla$ will be called a quotient of the CQG $(\clq,\Delta)$ by a quantum subgroup $H$ if $\cla$ is $C^*$-algebra isomorphic to the algebra
$\{x\in \clq~:(\pi\ot id)\Delta(x)=1\ot x\},$ where $\pi:\clq\rightarrow H$ is the CQG morphism. 

\bthm\label{3-equivalent conditions}
Let $l:\clq_0\rightarrow\IC$ be the generator of a quantum Gaussian process on $\clq.$ Suppose $(\clq,\Delta)$ acts on a quantum space $\cla$ such that $\cla$ is a quotient space. Denote the action by $\alpha$ and define $\cll_{_{\cla}}:=(id_\cla\ot l)\alpha.$ Then the following conditions are equivalent:
\begin{enumerate}
\item $(l\ot id)\Delta=(id\ot l)\Delta.$
\item $(l\ot id_\clq)\alpha=(id_\cla\ot l)\alpha.$ 
\item $(\cll_{_{\cla}}\ot id_\clq)\alpha=\alpha\circ \cll_{_{\cla}}.$
\end{enumerate}
\ethm
\begin{proof}
It can be shown (see \cite{podles}) that $\alpha=\Delta|_{_{\cla}}$ in case of quotien spaces, where $\cla$ has been identified with the algebra $\{x\in \clq~:(\pi\ot id)\Delta(x)=1\ot x\}.$ Thus $(1)\Rightarrow(2)$ is trivial. Let us prove $(2)\Rightarrow(1).$ It can be shown (see \cite[page 5]{podles}) that if $\cla$ is a quotient space, then 
the subspaces $W_{\gamma i}$ for $\gamma\in Irr_{_{\clq}},$ as described in Proposition \ref{spectral subspaces of action} are spanned by $\{u^\gamma_{ij}\}_{j=1}^{d_\gamma}$ and cardinality of the set $I_\gamma$ is $n_\gamma.$ So for a fixed $i,j,~i=1,2,...n_\gamma;~j=1,2,....d_\gamma,$ we have 
\begin{equation*}
\begin{split}
&(l\ot id_\clq)\alpha(u^\gamma_{ij})=(id_\cla\ot l)\alpha(u^\gamma_{ij})\\
&i.e.\\
&\sum_{k=1}^{d_\gamma}l(u^\gamma_{ik})u^\gamma_{kj}=\sum_{k=1}^{d_\gamma}u^\gamma_{ik}l(u^\gamma_{kj})~;
\end{split}
\end{equation*}
comparing the coefficients, we get $l(u^\gamma_{ii})=l(u^\gamma_{jj}),~l(u^\gamma_{ij})=0$ for $i\neq j,$  where $1\leq i\leq n_\gamma$ and $1\leq j\leq d_\gamma.$ As a vector space, $\clq_0=\oplus_{\gamma\in Irr_{_{\clq}}}\oplus_{i=1}^{d_\gamma}W_{\gamma i}.$ From the preceding discussions, it follows that $(l\ot id)\Delta(u^\gamma_{ij})=(id\ot l)\Delta(u^\gamma_{ij})$ which implies that $(l\ot id)\Delta=(id\ot l)\Delta$ i.e. $(2)\Rightarrow(1).$ $(1)\Rightarrow(3)$ was already observed right after defining covariant quantum Gaussian process. The proof of the theorem will be completed if we show $(3)\Rightarrow(2).$ This can be argued as follows:

Since $\cla$ is a quotient, we have $\alpha=\Delta|_{_{\cla}}.$ Consider the functional $\epsilon|_{_{\cla_0}},$ where $\cla_0:=\cla\cap \clq_0.$ Note that $\epsilon|_{_{\cla_0}}\circ \cll_{_{\cla}}=l.$ So applying $\epsilon|_{_{\cla_0}}\ot id_\clq$ on both sides of $(3),$ we get $(l\ot id_\clq)\alpha=\cll_{_{\cla}}:=(id_\cla\ot l)\alpha.$ Thus $(3)\Rightarrow(2).$ 
\end{proof}

\subsection{Deformation of Quantum Brownian motion}
Recall the set-up and notations of section \ref{CQG}, where the Rieffel deformation of $(\clq,\Delta),$ denoted by $\clq_{\theta,-\theta},$ for some skew symmetric matrix $\theta,$ of a CQG was described. As a $C^*$-algebra, it is the fixed point subalgebra 
$(\clq\ot C^*(\IT^{2n}_\theta))^{\sigma\times\tau^{-1}},$ and has the same coalgebra structure as that of $\clq.$ 
\bthm\label{lifting of QBM}
Let $l$ be the generator of a quantum Gaussian process and $\cll:=\widetilde{l}.$ Suppose that $\cll\circ\sigma_{_{\overline{z}}}=\sigma_{_{\overline{z}}}\circ\cll,$ for $\overline{z}\in\IT^{2n}.$
Then we have the following: 
\begin{enumerate}
\item[(i)] $(\cll\ot id)((\clq_0\ota\clw)^{\sigma\times\tau^{-1}})\subseteq(\clq_0\ota \clw)^{\sigma\times\tau^{-1}};$
\item[(ii)] $\cll_\theta:=(\cll\ot id)|_{_{(\clq_0\ota\clw)^{\sigma\times\tau^{-1}}}}$ is a generator of a quantum Gaussian process;
\item[(iii)] with respect to the natural identification of $(\clq_{\theta,-\theta})_{-\theta,\theta}$ with $\clq,$
we have $(\cll_\theta)_{-\theta}=\cll.$
\end{enumerate}
\ethm
\begin{proof}
Notice that the counit $\epsilon$ and the coproduct $\Delta$ remains the same in the deformed algebra, as the coalgebra $\clq_0$ is vector space isomorphic to $(\clq_0\ota \clw)^{\sigma\times\tau^{-1}}.$ By our hypothesis, $\sigma_{_{\overline{z}}}\circ\cll=\cll\circ\sigma_{_{\overline{z}}},$ which implies $(i).$ 

Since $\cll$ is a CCP map, it follows that $\cll_\theta$ is a CCP map. Moreover, since we have the identity 
$l(abc)=l(ab)\epsilon(c)-\epsilon(ab)l(c)+l(bc)\epsilon(a)-\epsilon(bc)l(a)+l(ac)\epsilon(b)-\epsilon(ac)l(b)$ for $a,b,c\in\clq_0,$ it follows that $l_\theta:=\epsilon\circ\cll_\theta$ also satisfies the same identity on the coalgebra $(\clq_0\ota \clw)^{\sigma\times\tau^{-1}}.$ Thus $l_\theta,$ or equivalently $\cll_\theta,$ generates a quantum Gaussian process on $\clq_{\theta,-\theta},$ which proves $(ii).$

(iii) follows from the natural identification of $(\clq_{\theta,-\theta})_{-\theta,\theta}$ with $\clq$ and an application of the result in (ii).
\end{proof}
We have the following obvious corollary:

\begin{crlre}
For a bi-invariant quantum Gaussian process, the conclusion of Theorem \ref{lifting of QBM} hold.
\end{crlre}

Thus we have a $1-1$ correspondence given by $\cll\leftrightarrow\cll_\theta,$ between the set of quantum Gaussian processes on $\clq$ and $\clq_{\theta,-\theta}.$ In case $\clq$ is co-commutative, i.e. $\Sigma\circ\Delta=\Delta,$ where $\Sigma$ is the flip operation, it is easily seen that any quantum Gaussian process on $\clq$ will be bi-invariant and so the $1-1$ correspondence $\cll\leftrightarrow\cll_\theta$ holds for arbitrary quantum Gaussian processes in such a case. It is not clear, however, whether we can get $1-1$ correspondence between bi-invariant QBM on the deformed and undeformed CQGs.

\bthm\label{abelian Lie-algebra}
If in the setup of Theorem \ref{lifting of QBM}, we have $\clq=C(G)$ for a compact Lie-group $G$ with abelian Lie-algebra $\mathfrak{g},$ then the hypothesis of Theorem \ref{lifting of QBM} and hence the conclusion hold.
\ethm

\begin{proof}
Let $G=G_e\bigsqcup_{i\in\Lambda}G_i,$ where $e\in G$ is the identity element and $G_e,G_i$ are the connected components of $G,$ $G_e$ being the identity component. Let the coproduct of the Rieffel-deformed algebra $\clq_{\theta,-\theta}$ be denoted by $\Delta_\theta$ (note that it is the same coproduct as the original one). Observe that since the action $\sigma$ is strongly continuous, $\overline{z}\cdot G_e\subseteq G_e~\forall \overline{z}\in\IT^{2n},$ or equivalently, we have $\sigma_{_{\overline{z}}}(C(G_e))\subseteq C(G_e).$ Thus one has the following decomposition:
\begin{equation*}
(C(G))_{\theta,-\theta}:=(C(G_e))_{\theta,-\theta}\oplus(\clb)_{\theta,-\theta},
\end{equation*}
where $\clb:=\oplus_{i\in\Lambda}C(G_i)$ and $C(G_e)_{\theta,-\theta}$ itself is a quantum group satisfying $\Delta_\theta\left(C(G_e)_{\theta,-\theta}\right)\subseteq C(G_e)_{\theta,-\theta}\ot C(G_e)_{\theta,-\theta}.$ Note that since $G_e$ is an abelian Lie-group, $C(G_e)_{\theta,-\theta}$ is a co-commutative quantum group. We claim that $l$ is supported on $C(G_e)_{\theta,-\theta}.$ Observe that $\chi_{_{G_e}}$ (the indicator function of $G_e$)
$\in C(G_e).$ Moreover, we have $\sigma_{_{\overline{z}}}(\chi_{_{G_e}})=\chi_{_{G_e}}.$ Thus $\chi_{_{G_e}}$ is identified with $\chi_{_{G_e}}^\theta:=\chi_{_{G_e}}\ot1\in(C(G)\ot C^*(\IT^{2n}_\theta))^{\sigma\times\tau^{-1}}.$ In particular, $\chi^\theta_{_{G_e}}$ is a self-adjoint idempotent in $C(G)_{\theta,-\theta}.$ It now suffices to show that $l((1-\chi^\theta_{_{G_e}})a)=0$ for all 
$a\in(C(G)_0\ota \innerl U_i|i=1,2,...2n\innerr_\IC)^{\sigma\times\tau^{-1}}.$ Let $(l,\eta,\epsilon)$ be a Schurmann triple for $l.$ Now 
\begin{equation*}
l((1-\chi^\theta_{_{G_e}})a)=l(1-\chi^\theta_{_{G_e}})\epsilon(a)+\epsilon(1-\chi^\theta_{_{G_e}})l(a)+\innerl\eta(1-\chi^\theta_{_{G_e}}),\eta(a)\innerr.
\end{equation*}
Now as $(1-\chi^\theta_{_{G_e}})^2=(1-\chi^\theta_{_{G_e}}),$ and clearly $\epsilon(1-\chi^\theta_{_{G_e}})=0,$ which implies that $1-\chi^\theta_{_{G_e}}\in ker(\epsilon)^2.$ By conditions 2 and 6 of Proposition \ref{franz}, we have $l(1-\chi^\theta_{_{G_e}})=\eta(1-\chi^\theta_{_{G_e}})=0.$ This implies that $l((1-\chi^\theta_{_{G_e}})a)=0$ for all $a\in (C(G)_0\ota \innerl U_i|i=1,2,...2n\innerr_\IC)^{\sigma\times\tau^{-1}}.$ Now as $(C(G_e))_{\theta,-\theta}$ is a co-commutative quantum group, we have 
\begin{equation}\label{co-commutativity}
(l\ot id)\Delta_\theta=(id\ot l)\Delta_\theta~  \mbox{on}~C(G_e)_{\theta,-\theta}.
\end{equation}
Let $\overline{z}=(u,v)$ for $u,v\in\IT^n.$ Let us recall that 
$\sigma_{_{\overline{z}}}=(\Omega(u)\ot id)\Delta(id\ot\Omega(-v))\Delta,$ where we have $\Omega(u):=ev_{u}\circ\pi,$ $\pi:C(G)\rightarrow C(\IT^n)$ being the surjective CQG morphism. Let $R(x):=\sigma_{_{(0,x)}}$ and $L(x):=\sigma_{_{(x,0)}}$ for $x\in\IT^n.$ By equation (\ref{co-commutativity}), we have $l(R(u)a)=l(L(u)a)$ for all $a\in C(G_e)_{\theta,-\theta}.$ Now $L(u)(C(G_i)_{\theta,-\theta})\subseteq C(G_i)_{\theta,-\theta}$ and $R(u)(C(G_i)_{\theta,-\theta})\subseteq C(G_i)_{\theta,-\theta}$ for all $i$ and $l(C(G_i)_{\theta,-\theta})=0,$ which, in combination with equation (\ref{co-commutativity}), gives 
$l(R(u)a)=l(L(u)a)$ for all $a\in C(G)_{\theta,-\theta}.$ From this, it easily follows that 
$\cll\circ\sigma_{_{\overline{z}}}=\sigma_{_{\overline{z}}}\circ\cll$ for all $\overline{z}\in\IT^{2n}.$
\end{proof}
Moreover, in subsection \ref{computation of QBM}, we shall see that condition of Theorem \ref{lifting of QBM} is indeed necessary, i.e. there may not be a `deformation' of a general quantum Gaussian generator.

\subsection{Computation of Quantum Brownian motion}\label{computation of QBM}
In this subsection, we compute the generators of QBM on the QISO of various non-commutative manifolds. We refer the reader to subsection \ref{QISO} for a recollection of the description of QISO of the non-commutative manifolds which we will consider here.
\begin{enumerate}
\item[a.] {\bf non-commutative $2$-tori:} Recall from subsection \ref{QISO} that $C^*(\IT^2_\theta)$ is the universal $C^*$-algebra generated by a pair of unitaries $U,V$ satisfying the relation $UV=e^{2\pi i\theta}VU.$
The QISO of $C^*(\IT^2_\theta)$ is a Rieffel deformation of the compact quantum group\\ $C\left(\IT^2\rtimes(\IZ_2^2\rtimes \IZ_2)\right)$ (see \cite{jyoti}). Moreover, $\IT^2\rtimes(\IZ_2^2\rtimes \IZ_2)$ is a Lie-group with abelian Lie-algebra. Hence an application of Theorem \ref{abelian Lie-algebra} and Theorem \ref{lifting of QBM} leads to the conclusion that the generators of quantum Gaussian processes on the QISO of $C^*(\IT^2_\theta)$ are precisely those coming from QISO($C(\IT^2)$)=ISO($\IT^2$)$\cong C(\IT^2\rtimes(\IZ_2^2\rtimes \IZ_2))$ i.e. they are of the form $l_\theta,$ where $l$ is a generator of classical Gaussian process on $\IT^2\rtimes(\IZ_2^2\rtimes \IZ_2),$ i.e. on its identity component $\IT^2.$ It can be seen by a direct computation that the space of $\epsilon$-derivations on QISO($C^*(\IT^2_\theta)$) is same as the space of $\epsilon$-derivations on $C(\IT^2\rtimes(\IZ_2^2\rtimes \IZ_2)).$ Moreover, all the $\epsilon$-derivations are supported on the identity component namely $C(\IT^2),$ which remains undeformed as a quantum subgroup of QISO($C^*(\IT^2_\theta)$). Thus it follows that in this case, a QBM on the undeformed CQG remains a QBM on the deformed CQG.

Using the action $\alpha$ as described in subsection \ref{QISO}, we can construct a QBM on $C^*(\IT^2_\theta)$ as described in section \ref{notun section}, and conclude that 
\bthm
Any QBM $k_t$ on $C^*(\IT^2_\theta)$ is essentially driven by a classical Brownian motion on $\IT^2,$ in the sense that $k_t:C^*(\IT^2_\theta)\rightarrow C^*(\IT^2_\theta)^{\prime\prime}\ot B(\Gamma(L^2(\IR_+,\IC^2)))\cong B(L^2(\omega_1,\omega_2)),$ where $(\omega_1,\omega_2)$ is the 2-dimensional standard Wiener measure, is given by $k_t(a)(\omega_1,\omega_2)=\alpha_{_{(e^{2\pi i \omega_1},e^{2\pi i \omega_2})}}(a).$
\ethm 

We now give an intrinsic characterization of a qunatum Gaussian (QBM) generator on $C^*(\IT^2_\theta):$

Let $\cla_0$ denote the $\ast$-subalgebra spanned by the unitaries $U,V.$ 

\bthm
A linear CCP map $\cll:\cla_0\rightarrow\cla_0$ is a generator of a quantum Gaussian process (QBM) on $C^*(\IT^2_\theta)$ if and only if $\cll$ satisfies:
\begin{enumerate}
\item[1.] $\cll(abc)=\cll(ab)c-ab\cll(c)+\cll(bc)a-bc\cll(a)+\cll(ac)b-ac\cll(b),$ for all $a,b,c\in\cla_0.$
\item[2.] $(\cll\ot id)\circ\alpha=\alpha\circ\cll,$ where $\alpha$ is the action of $\IT^2$ on $C^*(\IT^2_\theta).$\\ Moreover, $\cll$ will generate a QBM if and only if $$l_{(1,1)}-l_{(1,0)}-l_{(0,1)}<2\sqrt{Re(l_{(1,0)})Re(l_{(0,1)})},$$ where $l(U)=l_{_{(1,0)}}U,~l(V):=l_{_{(0,1)}},~l(UV):=l_{_{(1,1)}}UV.$

\end{enumerate}
\ethm

\begin{proof}
Suppose that $\cll$ is the generator of a quantum Gaussian process (QBM) on $C^*(\IT^2_\theta).$ Notice that condition (2.) implies that $U,V,UV$ are the eigenvectors of $\cll.$ Let the eigenvalues be denoted by $l_{_{(1,0)}},l_{_{(0,1)}},l_{_{(1,1)}}$ respectively. Then there exists a Gaussian (Brownian) functional $l$ on QISO($C^*(\IT^2_\theta)$)($=\clq$) with surjective Schurmann triple $(l,\eta,\epsilon),$ such that $\cll=(id\ot l)\alpha.$  Let $(\eta_i)_{i=1,2}$ be the coordinates of $\eta.$ Then since 
$l(abc)=l(ab)\epsilon(c)-\epsilon(ab)l(c)+l(bc)\epsilon(a)-\epsilon(bc)l(a)+l(ac)\epsilon(b)-\epsilon(ac)l(b)$ for $a,b,c\in\clq_0,$ we have condition 1. of the present theorem. Condition (2.) follows by a direct computation, along with the fact that if $l$ is generates a QBM, then $\eta_1,\eta_2$ spans the space $\clv_{C^*(\IT^2_\theta)}.$

Conversely, suppose that we are given a CCP functional $\cll,$ satisfying conditions $(1.)$ and ($2.$). Choose two vectors $(c_1,c_2),(d_1,d_2)\in\IR^2$ such that $c_1^2+c_2^2=-2Re(l_{(1,0)}),$ $d_1^2+d_2^2=-2Re(l_{(0,1)}),$ and $c_1d_1+c_2d_2=l_{(1,1)}-l_{(1,0)}-l_{(0,1)}.$ Consider the two $\epsilon$-derivations $\eta_1:=c_1\eta_{(1)}+d_1\eta_{(2)}$ and $\eta_2:=c_2\eta_{(1)}+d_2\eta_{(2)}.$ Define a CCP finctional $l_{new}$ on $\clq$ as :
$l_{new}(U_{11})=l_{(1,0)}$ and $l_{new}(U_{12})=l_{(0,1)},$ $l_{new}(U_{kj})=0$ for $k>1,j=1,2,$ and extend the definition to $(\clq)_0$ by the rule $l(a^*b)=l(a^*)\epsilon(b)+l(b)\epsilon(a^*)+\sum_{p=1}^2\overline{\eta_{_{1}}(a^*)}\eta_{_{p}}(b).$ Note that we have $l_{new}(abc)=l_{new}(ab)\epsilon(c)-\epsilon(ab)l_{new}(c)+l_{new}(bc)\epsilon(a)-\epsilon(bc)l_{new}(a)+l_{new}(ac)\epsilon(b)-\epsilon(ac)l_{new}(b)$ for $a,b,c\in\clq_0.$ It follows that $\cll_{new}:=(id\ot\cll_{new})\alpha$ satisfies conditions (1.) and (2.) Thus $\cll=\cll_{new}$ on $\cla_0$ and since $\cll_{new}$ generates a quantum Gaussian process (QBM) on $C^*(\IT^2_\theta),$ so does $\cll.$
\end{proof}
\begin{rmrk}
It follows from this that in a similar way, we can also characterize generators of quantum Gaussian processes on quantum spaces on which $\IT^n$ acts ergodically.
\end{rmrk}

\item[b.] {\bf The $\theta$ deformed sphere $S^n_\theta$:} 
\bthm
\begin{enumerate}
\item[(i)] Suppose that $l$ is the generator of a  quantum Gaussian process on $O_\theta(2n).$ Then it satisfies the following:

There exists $2n$ complex numbers $\{z_1,z_2,.....z_{2n}\}$ with $Re(z_i)\leq0$ for all $i$ and ${\bf A}\in M_{2n}(\IC)$  with  $A_{ii}=0~\forall~i$ and $[A_{ij}-z_i-z_j]_{ij}\geq0,$ such that 
\begin{equation}\label{ekmatra condition}
l(a^i_i)=z_i,~l(a^{i\ast}_ia^j_j)=A_{ij}~i,j=1,2,....2n.
\end{equation}

Conversely, given $2n$ complex numbers $\{z_1,z_2,.....z_{2n}\}$ and ${\bf A}\in M_{2n}(\IC),$  such that $Re(z_i)\leq0,~A_{ii}=0~\forall~i$ and $[A_{ij}-\overline{z_i}-z_j]_{ij}\geq0,$ there exists a unique map $l,$ such that $l$ generates a quantum Gaussian process and satisfies equation (\ref{ekmatra condition}).
\item[(ii)] The generator of a quantum Gaussian process say $l$ generates a QBM if and only if the matrix 
\begin{equation}
\left[l(a^{\mu\ast}_\mu a^\nu_\nu)-l(a^{\mu\ast}_\mu)-l(a^\nu_\nu)\right]_{\mu,\nu}\in M_{2n}(\IC)
\end{equation}
is invertible.

\item[(iii)] $l$ generates a bi-invariant quantum Gaussian process if and only if $z_\beta=z$ for all $\beta=1,2,..2n,$ where $z\in\IR$ such that $z\leq0.$
\end{enumerate}
\ethm
\begin{proof}
Let us first calculate all possible $\epsilon$-derivations. Let $\eta$ be an $\epsilon$-derivation on this CQG. Put $\eta(a^\mu_\nu)=c^\mu_\nu,$ $\eta(\overline{a^\mu_\nu})=\widehat{c^\mu_\nu},$ $\eta(b^\mu_\nu)=d^\mu_\nu,$ $\eta(\overline{b^\mu_\nu})=\widehat{d^\mu_\nu},$ $\mu.\nu=1,2,...n.$ Using condition (a), we get 
\begin{equation*}
c^\mu_\nu\delta^\tau_\rho+c^\tau_\rho\delta^\mu_\nu=\lambda_{\mu\tau}\lambda_{\rho\nu}(c^\mu_\nu\delta^\tau_\rho+c^\tau_\rho\delta^\mu_\nu);
\end{equation*}
putting $\tau=\rho,$ we get $c^\mu_\nu=0$ for $\mu\neq\nu.$ Likewise using conditions (b) and (c), we get $\widehat{c}^\mu_\nu=d^\mu_\nu=\widehat{d}^\mu_\nu=0$ for $\mu\neq\nu.$ Using condition (d) with $\alpha=\beta,$ we arrive at the following relations:
\begin{equation*}
\begin{split}
&\widehat{c}^\alpha_\alpha+c^\alpha_\alpha=0~(since~\eta(1)=0),\\
&d^\alpha_\alpha+d^\alpha_\alpha=0,\\
&\widehat{d}^\alpha_\alpha+\widehat{d}^\alpha_\alpha=0;
\end{split}
\end{equation*}
this implies that $c^\alpha_\beta=c_\beta\delta_{\alpha\beta},$ $\widehat{c}^\alpha_\beta=-c_\beta\delta_{\alpha\beta}$ for $n$ complex numbers $\{c_1,c_2,...c_n\}.$ It may be noted that all the above steps are reversible, and hence this also characterizes $\epsilon$-derivations on $O_\theta(2n).$ Note that the space of $\epsilon$-derivations, $\clv_{O_\theta(2n)}$ is $2n$-dimensional and is spanned by $n$ $\epsilon$-derivations $\{\eta_{(1)},\eta_{(2)},...\eta_{(2n)}\},$ where $(\eta_{(k)}(a^\alpha_\beta))_{\alpha,\beta}=E_{kk},$ where $E_{ij}$ denote an elementary matrix.

Now we prove (i) as follows:

Let $l$ be the generator of a quantum Gaussian process. Let the surjective Schurmann triple of $l$ be $(l,\eta,\epsilon).$ Let $(\eta_i)_i$ be the coordinates of $\eta,$ which are $\epsilon$-derivations. By Lemma \ref{bounded by the noise space}, there can be atmost $2n$ such coordinates. Let $\eta_i(a^\alpha_\beta)=c^{(i)}_{\alpha\beta}$ and $\eta_i(a^{\alpha\ast}_\beta)=\widehat{c}^{(i)}_{\alpha\beta}$ such that $c^{(i)}_{\alpha\beta}=c^{(i)}_\beta\delta_{\alpha\beta}$ and $\widehat{c}^{(i)}_{\alpha\beta}=-c^{(i)}_\beta\delta_{\alpha\beta}.$ Suppose that $l(a^\alpha_\beta)=l_{\alpha\beta}$ and $l(b^\alpha_\beta)=l^\prime_{\alpha\beta}.$ Then using the relations among the generators of $O_\theta(2n),$ as given in subsection \ref{QISO}, we arrive at the following results:

\begin{equation*}
\begin{split}
& l_{\alpha\beta}=0~\mbox{for all $\alpha\neq\beta$},\\
& \overline{l}_{\alpha\alpha}+l_{\alpha\alpha}=-\sum_i|c_\alpha^{(i)}|^2~\mbox{for all $\alpha=1,2,...2n$},\\
& l^{\prime}_{\alpha\beta}=0 ~\mbox{for all $\alpha,\beta$}.
\end{split}
\end{equation*}
Moreover, we have $l(a^*b)-l(a^*)\epsilon(b)-\epsilon(a^*)l(b)=\innerl\eta(a),\eta(b)\innerr,$ so that by taking 
$z_i:=l(a^i_i),~{\bf A}:=[l(x^{i\ast}_ia^j_j)]_{ij}$  we have the result. 

Conversely, suppose that we are given $2n$ complex numbers $\{z_1,z_2,...z_{2n}\}$ such that $Re(z_i)\leq0$ for all $i$ and ${\bf A}\in M_{2n}(\IC),$ satisfying the hypothesis. Let ${\bf B}:=[A_{ij}-\overline{z_i}-z_j]_{ij}.$ Suppose that  ${\bf P}:={\bf B}^{\frac{1}{2}}.$ Let us define $2n$ $\epsilon$-derivations $(\eta_i)_{i=1}^{2n}$ by 
$\eta_k:=\sum_{i=1}^{2n}\overline{P_{ik}}\eta_{_{(i)}},$ $k=1,2,....2n.$ Let $\eta:=\sum_{i=1}^{2n}\eta_i\ot e_i,$ where $\{e_i\}_i$ is the standard basis of $\IC^{2n}.$ Define a CCP map $l$ on $O_\theta^{alg}(2n)$ by the prescription $l(a^i_i)=z_i,l(a^i_j)=0~\mbox{for}~i\neq j,$ $l(b^i_j)=0~\forall~ i,j$ and extending the map to $O_\theta^{alg}(2n)$ by the rule $l(a^*b)=l(a^*)\epsilon(b)+\epsilon(a^*)l(b)+\innerl\eta(a),\eta(b)\innerr.$ Such a map is clearly the generator of a quantum Gaussian process on $O_\theta(2n)$ and it satisfies $l(a^{i\ast}_ia^j_j)=A_{ij}.$ The uniqueness follows from the fact that a generator of a quantum Gaussian process on $O_\theta(2n)$ must satisfy the identity:
$$l(abc)=l(ab)\epsilon(c)-\epsilon(ab)l(c)+l(bc)\epsilon(a)-\epsilon(bc)l(a)+l(ac)\epsilon(b)-\epsilon(ac)l(b),$$ for all $a,b,c\in O_\theta^{alg}(2n).$ 

For proving (ii), let us proceed as follows:

Let $l$ be the generator of a QBM and let $(l,\eta,\epsilon)$ be the surjective Schurmann triple associated with $l.$
Suppose that $(\eta_i)_i$ are the coordinates of $\eta.$ Then by hypothesis, $\{\eta_1,\eta_2,...\eta_{2n}\}$ forms a basis for $\clv.$ Let $\eta_k=\sum_i c^{(k)}_i\eta_{(i)}.$ Consider the $2n\times 2n$ matrix ${\bf P}$ such that 
$P_{ij}:=\overline{c^{(j)}_i}.$ Then ${\bf P^*P}$ is an invertible matrix. Moreover, we have $[l(a^{i\ast}_ia^j_j)-l(a^{i\ast}_i)-l(a^j_j)]_{ij}={\bf P^*P},$ which proves our claim. 

Conversely, suppose that $l$ is the generator of a quantum Gaussian process, such that ${\bf B}:=[l(a^{i\ast}_ia^j_j)-l(a^{i\ast}_i)-l(a^j_j)]_{ij}$ is an invertible matrix. Let $(l,\eta,\epsilon)$ be the surjective Schurmann triple associated to $l.$ Let $(\eta_i)_i$ be the coordinates of $\eta.$ Let 
$\eta_k=\sum_i c^{(k)}_i\eta_{(i)},$ for all $k.$ Let ${\bf P}:=[c^{(j)}_i]_{ij}.$ Then we have ${\bf P^*P}={\bf B},$ which implies that the matrix ${\bf P}$ is invertible, and hence $\{\eta_i\}_{i=1}^{2n}$ forms a basis for $\clv_{_{O_\theta(2n)}},$ which proves the claim. 

(iii) follows by a direct computation using the formula for coproduct, as given in subsection \ref{QISO}.
\end{proof}
We have the following obvious corollary, which follows from $(iii)$ of the theorem above and the definition of quantum Gaussian process on quantum homogeneous space.
\begin{crlre}
A map $\cll_{_{S^{2n-1}_\theta}},$ which generates a qunatum Gaussian process on $S^{2n-1}_\theta,$ satisfy:
$\cll_{_{S^{2n-1}_\theta}}(z^\mu)=cz^\mu,$ for some real number $c\leq0.$
\end{crlre}
\begin{rmrk}
Notice that the space of $\epsilon$-derivations on the undeformed algebra $O(2n)$ has dimension more than $2n,$ since there are $\epsilon$-derivations, which takes non-zero values on $(b^\mu_\nu)_{\mu\nu}$ and hence there are quantum Gaussian processes on $O(2n)$ such that their generators take non-zero values on $b^\alpha_\beta,$ and so there is no 1-1 correspondence between quantum Gaussian processes on the deformed and undeformed algebra in this case. 
\end{rmrk}

\item [c.]{\bf The free orthogonal group  $O_+(2n)$}: 
We refer the reader to subsection \ref{QISO} again, for the definition and formulae for the free orthogonal group. Before stating the main theorem, we introduce some notations for convenience. Let 
${\bf A}\in M_{n(2n-1)}(\IC).$ We will index the elements of ${\bf A}$ by the set $\IN^4$ instead of $\IN^2$ as follows:

Let ${\bf A}=\begin{pmatrix}
{\bf A_1}\\            
{\bf A_2}\\
.\\
.\\
.\\
{\bf A_{2n-1}}
\end{pmatrix},$ where ${\bf A_i}$ is a $n(2n-1)\times(2n-i+1)$ matrix, such that
\begin{equation*}
\begin{split}
{\bf (A_i)_{kl}}=&
a_{_{(i,i+k,1,1+l)}}\chi_{_{\{1,2,...2n-1\}}}(l)\\
&+a_{_{(i,i+k,2,3+(l-2n))}}\chi_{_{\{2n,2n+1,...4n-3\}}}(l)\\
&+a_{_{(i,i+k,3,4+(l-(4n-2)))}}\chi_{_{\{4n-2,...6n-2\}}}(l)\\
&+\\
&.\\
&.\\
&.\\
&+a_{_{(i,i+k,2n-1,2n)}}\chi_{_{\{n(2n-1)\}}}(l),
\end{split}
\end{equation*}
for $k=1,2,...2n-i+1,$ where $\chi_{_{B}}$ denotes the indicator function of the set $B.$ We now state the main theorem:
\bthm
\begin{enumerate}
\item[(i)]There exists a 1-1 correspondence between generators of quantum Gaussian processes on $O_+(2n)$ and matrices ${\bf L}:=[L_{ij}]\in M_{2n}(\IC)$ and ${\bf A}:=[A_{ij}]\in M_{n(2n-1)}(\IC),$ satisfying
\begin{enumerate}
\item[{\bf a.}]~${\bf B}\in M_{n(2n-1)}(\IC),$ defined by ${\bf B}:=[a_{_{(i,j,k,l)}}-\overline{L_{ij}}-L_{kl}],~i<j,~k<l$ is positive definite,
\item[{\bf b.}]~$L_{ij}+L_{ji}:=-\sum_{k=1}^{i-1}a_{_{(k,i,k,j)}}+\sum_{k=i+1}^{j-1}a_{_{(i,k,k,j)}}-\sum_{k=j+1}^{2n}a_{_{(i,k,j,k)}},~i<j.$ 
\end{enumerate}
\item[(ii)] $l$ will generate a QBM if and only if the matrix ${\bf B},$ defined above, is invertible.
\item[(iii)] There exists no bi-invariant quantum Gaussian process on $O_+(2n).$
\end{enumerate}
\ethm
\begin{proof}
Using the relations among the generators, as given in subsection \ref{QISO}, it is seen that the epsilon-derivations on this algebra are given by 
\begin{equation*}
\eta(x_{ij})=A_{ij};
\end{equation*}
such that $A_{ij}=-A_{ji}.$ Clearly this characterizes the $\epsilon$-derivations on the CQG. Observe that the space of $\epsilon$-derivations, $\clv_{O_+(2n)}$ has dimension $n(2n-1).$ A basis for the space is given by 
$\{\eta_{(ij)}\}_{i<j},$ such that $\eta_{(ij)}(x_{ij})=1,~\eta_{(ij)}(x_{ji})=-1$ and $\eta_{(ij)}(x_{kl})=0$ for $k\neq i,j$ or $l\neq i,j.$ So after a suitable re-indexing, let us denote the basis by $\{\eta_{_{(p)}}\}_{p=1}^{n(2n-1)}.$

We prove (i):

Let $l$ be the generator of a quantum Gaussian process on $O_\theta(2n),$ with the surjective Schurmann triple
$(l,\eta,\epsilon).$  Let $(\eta_i)_i$ be the coordinates of $\eta.$ By Lemma \ref{bounded by the noise space}, there can be atmost $n(2n-1)$ coordinates. Let ${\bf A^{(i)}}:=((\eta_i(x_{kl})))_{kl}.$ Now using the relations among the generators, as described in subsection \ref{QISO}, we see that
\begin{equation*}
((l(x_{ij})))_{i,j}={\bf L}
\end{equation*}
such that 
\begin{equation*}
{\bf L}_{ij}+{\bf L_{ji}}=-\sum_{s=1}^{n(2n-1)}\sum_{k=1}^{2n}\overline{A^{(s)}_{ik}}A^{(s)}_{jk},~i<j.
\end{equation*}
Thus by taking $a_{_{(i,j,k,l)}}:=l(x_{ij}x_{kl})$ and $ L_{ij}:=l(x_{ij}),$ the conclusion follows. 

Conversely, suppose that we are given matrices ${\bf L}\in M_{2n}(\IC),~{\bf A}\in M_{n(2n-1)}(\IC)$ satisfying the hypothesis in (i). Let ${\bf P}:={\bf B}^{\frac{1}{2}}.$ Define $n(2n-1)$ $\epsilon$-derivations by $\eta_{_{p}}:=\sum_{k=1}^{n(2n-1)}{\bf P_{pk}}\eta_{_{(k)}},$\\ $p=1,2,...n(2n-1).$ Define a CCP map by the prescription $l(x_{ij}):= L_{ij},$ and extending the definition to $O_+^{alg}(2n)$ by the rule $l(a^*b)=l(a^*)\epsilon(b)+\epsilon(a^*)l(b)+\sum_{p=1}^{n(2n-1)}\overline{\eta_{_{p}}(a)}\eta_{_{p}}(b),$ where $O^{alg}_+(2n)$ is the $\ast$-algebra generated by $x_{ij},i,j=1,2,...n(2n-1).$ Clearly such a functional satisfies $l(x_{ij}x_{kl})=a_{_{(i,j,k,l)}},~i<j,~k<l.$ The uniqueness follows from the fact that $l$ satisfies $l(abc)=l(ab)\epsilon(c)-\epsilon(ab)l(c)+l(bc)\epsilon(a)-\epsilon(bc)l(a)+l(ac)\epsilon(b)-\epsilon(ac)l(b),~a,b,c\in O^{alg}_+(2n).$

(ii) follows from the fact that the invertibility of the matrix ${\bf B}$ implies the invertibility of the matrix ${\bf P}:={\bf B}^{\frac{1}{2}},$ so that $\{\eta_{_{i}}\}_{i=1}^{n(2n-1)},$ as defined in (i), forms a basis for $\clv_{_{O_+(2n)}}.$

(iii) can be proven as follows:

\bthm\label{non-existence of bi-invariant QBM}
Suppose $\cll$ is the generator of a bi-invariant QBM on the free orthogonal group. Then $\cll\equiv0.$
\ethm
\begin{proof}
Since $\cll$ is bi-invariant, we have 

\begin{equation}\label{first equation of bi-invariance}
(id\ot\cll)\Delta(x_{ij})=(\cll\ot id)\Delta(x_{ij})
\end{equation}
and 
\begin{equation}\label{second equation of bi-invariance}
(id\ot\cll)\Delta(x_{ij}x_{kl})=(\cll\ot id)\Delta(x_{ij}x_{kl})~\mbox{where $i\neq j$ and $k\neq l$};
\end{equation}
comparing the coefficients in (\ref{first equation of bi-invariance}) and (\ref{second equation of bi-invariance}), we get 
\begin{equation*}
\begin{split}
&\cll(x_{ij})=0~\mbox{for $i\neq j$};\\
&\cll(x_{ij}x_{kl})=0~\mbox{for $i\neq j$ and $k\neq l$};
\end{split}
\end{equation*}
substituting $k=i,~l=j,(i\neq j)$ in the second equation, we get 
\begin{equation*}
0=\cll(x_{ij}x_{kl})=\sum_{p\geq1}\overline{\eta_{_{(p)}}(x_{ij})}\eta_{_{(p)}}(x_{ij})=\sum_{p\geq1}|\eta_{_{(p)}}(x_{ij})|^2,~i\neq j;
\end{equation*}
where $\eta_{_{p}}$ is an $\epsilon$-derivation for each $p.$  This implies $\eta_{_{(p)}}\equiv0,$ since $\eta_{_{(p)}}(x_{ii})=0.$ Thus $\cll$ becomes an $\epsilon$-derivation. But $\cll(x_{ij})=0$ for $i\neq j.$ Thus we have $\cll\equiv0.$ 
\end{proof}
\begin{rmrk}
Theorem \ref{non-existence of bi-invariant QBM} implies that there does not exists any quantum Brownian motion on the quantum space $S_{2n-1}^+$ (i.e. the free sphere) in the sense described in subsection \ref{notun section}.
\end{rmrk}
\end{proof}

\end{enumerate}

\section{Exit time of Quantum Brownian motion on non-commutative torus.}
\subsection{Motivation and formulation}
We shall first recast the classical results about the assymptotics of exit time of Brownian motion in a form which will be easily generalized to the quantum set-up.

Let $M$ be a Riemannian manifold of Dimension $d$ which is also a homogeneous space. Therefore $M$ can be realized as $K/G$, where $G$ is the isometry group of $M$ and $K$ is a compact subgroup of $G$. For $m\in M,$ let $\clb_t^m$ denote the standard Brownian motion on $M$ starting at $m$, as described in section \ref{notun section}. Let $\wA$ denote the universal enveloping von-Neumann algebra of $C(M)$. Let us define a map $j_t:\wA\rightarrow\wA\ot B(L^2(\IP))$ by:
$j_t(f)(x,\omega):=f(\clb_t^x(\omega))$, for $f\in C(M)$ and extending the map to $\wA$, where $\IP$ denote the $d$-dimensional Wiener measure. 

Let $B^x_r$ denote a ball of radius $r$ around $x\in M.$ Let $\tau_{_{B^x_r}}$ be the exit time of the Brownian motion from the ball $B^x_r.$ Then $\{\tau_{_{B^x_r}}>t\}=\{\clb_s^x\in B^x_r\forall~0\leq s\leq t\},$ so that we have\\ $\chi_{_{\{\tau_{B^x_r}>t\}}}=\bigwedge_{s\leq t}\left(\chi_{_{\{\clb_s^x\in B^x_r\}}}\right),$ where $\bigwedge$ denotes infimum and for a set $A$, $\chi_A$ denotes the indicator function on the set $A$. In terms of the map $j_t$, we have $$\chi_{_{\{\tau_{_{B^x_r}}>t\}}}(\cdot)=\bigwedge_{s\leq t}~j_s(\chi_{_{B^x_r}})(x,\cdot)=\bigwedge_{s\leq t}((ev_x\ot id)\circ j_s(\chi_{_{B^x_r}}))(\cdot).$$  
Now by the Wiener-It$\hat{o}$ isomorphism (see \cite{krp}), $L^2(\IP)\wtil{=}\Gamma(L^2(\IR_+,\IC^d)).$ Thus we may view $\tau_{_{B^x_r}}$ as a family of projections in $\wA\ot B(\Gamma(L^2(\IR_+,\IC^d)))$ defined by $$\tau_{_{B^x_r}}\left([0,t)\right)={\bf 1}-\wedge_{s\leq t}( j_s(\chi_{_{B^x_r}}))~.$$ 
\paragraph{}
We recall from subsection \ref{Classical Brownian motion}, the asymptotic behaviour of $\IE(\tau_{_{B^x_r}})$ as $r\rightarrow0.$
Now one has\\ $\IE(\tau_{_{B^x_r}})=\int_0^\infty\IP(\tau_{_{B^x_r}}>t)dt=\int_0^\infty\innerl e(0),\{(ev_x\ot1)\left(\wedge_{s\leq t} j_s(\chi_{_{B^x_r}})\right)\}e(0)\innerr dt,$ since $\tau_{_{B^x_r}}$ is a positive random variable. Note that the points of $M$ are in $1-1$ correspondence with the pure states and $\{P_r=\chi_{B^x_r}\}_{r\geq0}$ is a family of projections on $\wA$ satisfying $vol(P_r)\rightarrow0$ as $r\rightarrow0$ and $ev_x(P_r)=1~\forall r.$ One can slightly generalize this as follows:

Choose a sequence $(x_n)_n\in M$ and positive numbers $\epsilon_n$ such that $x_n\rightarrow x$ and $\epsilon_n\rightarrow0.$ Now for large $n_0$ the random variable $\chi_{_{\{\clb_s^{x_n}\in B^{x_n}_{\epsilon_n}\}}}(\cdot)$ has the same distributioin as the random variable $\chi_{_{\{\clb^x_s\in B^x_{\epsilon_n}\}}}$ for each $s\geq0.$ 
Thus, 
$$\IE(\tau_{_{B_{\epsilon_n}^{x_n}}})=\IE(\tau_{_{B^x_{\epsilon_n}}})=\int_0^\infty\innerl e(0),\{(ev_{x_n}\ot id)\left(\wedge_{s\leq t} j_s(\chi_{_{B^{x_n}_{\epsilon_n}}})\right)\}e(0)\innerr dt$$ which implies that the asymptotic behaviour of $\IE(\tau_{_{B_{\epsilon_n}^{x_n}}})$ and $\IE(\tau_{_{B^x_{\epsilon_n}}})$ will be the same.

For a non-commutative generalization of the above, we need the notion of quantum stop time. There are several formulations of this concept \cite{attal-kbs,krp-kbs,barnett}. The one most suitable for us is the following:

\bdfn\cite{barnett}\label{qstop}[Barnette]
Let $(\mathfrak{A}_t)_{t\geq0}$ be an increasing family of von-Neumann algebras (called a filtration). A quantum random time or stop time adapted to the filtration $(\mathfrak{A}_t)_{t\geq0}$ is an increasing family of projections $(E_t)_{t\geq0},~E_\infty=I$ such that $E_t$ is a projection in $\mathfrak{A}_t$ and $E_s\leq E_t$ whenever $0\leq s\leq t<+\infty.$ Furthermore, for $t\geq s,$ 
$E_t\downarrow E_s$ as $t\downarrow s.$
\edfn

Observe that by our definition, $\tau_{_{B_r}}([0,t))$ is adapted to the  filtration $(\mathfrak{A}_t)_{t\geq0},$ where\\ $\mathfrak{A}_t:=\wA\ot B(\Gamma_{t]})$ $\left(\Gamma_{t]}:=\Gamma\left(L^2([0,t],\IC^n)\right)~\right)$, for $\tau_{_{B_r}}([0,t])\in\mathfrak{A}_t\ot1_{\Gamma_{[t}}.$ 

Suppose that we are given an E-H flow $j_t:\cla\rightarrow\cla^{\prime\prime}\ot B(\Gamma(L^2(\IR_+,k_0))),$ where $\cla$ is a $C^*$ or von-Neumann algebra. For a projection $P\in\cla,$ the family $\{{\bf 1}-\wedge_{s\leq t}\left(j_s(P)\right)\}_{t\geq0}$ defines a quantum random time adapted to the filtration $\left(\cla^{\prime\prime}\ot B(\Gamma_{t]})\right)_{t\geq0}.$ Let us assume, furthermore, that $\cla$ is the $C^*$ or von-Neumann
closure of the `smooth algebra' $\cla^\infty$ of a $\Theta$-summable, admissible spectral triple and $j_t$ is a QBM on it.
\bdfn
We refer to the quantum random time $\{1-\bigwedge_{s\leq t}j_s(P)\}_{t\geq0}$ as the `exit time from the projection $P$'.
\edfn
Motivated by the Propostion \ref{pinsky} and the discussion after it, we would like to formulate a quantum analogue of the exit time asymptotics and study it in concrete examples.

Let $\tau$ be the non-commutative volume form corresponding to the spectral triple, and assume that we are given a family $\{P_n\}_{n\geq1}$ of projections in $\cla,$ and a family $\{\omega_n\}_{n\geq1}$ of pure states of $\cla$ such that 
\begin{itemize}
\item $\omega_n$ is weak$^\ast$ convergent to a pure state $\omega,$
\item $\omega_n(P_n)=1$ for all $n,$ 
\item $v_n\equiv\tau(P_n)\rightarrow0$ as $n\rightarrow\infty.$
\end{itemize}

\bdfn
Let $\gamma_n:=\int_0^\infty dt \innerl e(0),(\omega_n\ot id)\circ\bigwedge_{s\leq t}j_s(P_n)e(0)\innerr.$ We say that there is an exit time asymptotic for the family $\{P_n;\omega_n\}$ of intrinsic dimension $n_0$ if 
$$\lim_{n\rightarrow\infty}\frac{\gamma_n}{v_n^{\frac{2}{m}}}=
\begin{cases}
\infty~\mbox{if $m$ is just less than $n_0$}\\
\neq0~\mbox{if $m\neq n$}\\
=0~\mbox{if $m>n$}
\end{cases}
$$ 
and 
\begin{equation}\label{assymptotic expansion}
\gamma_n=c_1v_n^{\frac{2}{n_0}}+c_2 v_n^{\frac{4}{n_0}}+\cdot\cdot\cdot c_kv_n^{\frac{2^k}{n_0}}+O(v_n^{\frac{2^{k+1}}{n_0}})~\mbox{as}~n\rightarrow\infty.
\end{equation}
\edfn

It is not at all clear whether such an asymptotic exists in general, and even if it exists, whether it is independent of the choice of the family $\{P_n;\omega_n\}.$ If it is the case, one may legitimately think of $c_1,c_2$ as geometric invariants and imitating the classical formulae (\ref{intrinsic dimension}) and (\ref{formula for mean curvature}), the extrinsic dimension $d$ and the mean curvature $H$ of the non-commutative manifold may be defined to be 
\begin{equation}\label{extrinsic dimension}
d:=\frac{1}{2c_1}(\frac{n_0}{\alpha_{n_0}})^{\frac{2}{n_0}}+1,
\end{equation}
\begin{equation}\label{formula for mean curvature1}
H^2:=8(d+1)c_2(\frac{\alpha_{n_0}}{n_0})^{\frac{4}{n_0}}.
\end{equation}
\subsection{A case-study: non-commutative Torus.}
Fix an irrational number $\theta\in [0,1].$
We refer the reader to [\cite{davidson},page 173] for a natural class of projections in $C^*(\IT^2_\theta),$ which we will be using in this section.
\paragraph{}
Let $tr$ be the canonical trace in $C^*(\IT^2_\theta),$ given by $tr(\sum_{m,n}a_{mn}U^mV^n)=a_{00}.$ This trace will be taken as an analogue of the volume form in $C^*(\IT^2_\theta).$ Throughout the section, we will consider $C^*(\IT^2_\theta)$ as a concrete $C^*$-subalgebra of $B(\wH),$ where $\wH$ denote the so-called universal enveloping Hilbert space for $C^*(\IT^2_\theta),$ and let $W^*(\IT^2_\theta)$ be the universal enveloping von-Neumann algebra of it. i.e. the weak closure of $C^*(\IT^2_\theta)$ in $B(\wH).$ For $(x,y)\in\IT^2,$ let $\alpha_{_{(x,y)}}$ denote the canonical action of $\IT^2$ on $C^*(\IT^2_\theta)$ given by $\alpha_{_{(x,y)}}(\sum_{m,n}a_{mn}U^mV^n)=\sum_{m,n}x^my^na_{mn}U^mV^n.$ For a projection $P,$ let $A_{(t,s)}(P):=A_{_{s,t}}(P).$
Note that each $\alpha_{(x,y)}$ extends as a normal automorpihsm of $W^*(\IT^2_\theta).$ On $C^*(\IT^2_\theta),$ there are two conditional expectations denoted by $\phi_1,\phi_2,$ which are defined as:
\begin{equation*}
\phi_1(A):=\int_0^1\alpha_{_{(1,e^{2\pi it})}}(A)dt,~~\phi_2(A):=\int_0^1\alpha_{_{(e^{2\pi it},1)}}(A)dt.
\end{equation*}
By universality of $W^*(\IT^2_\theta),$ $\phi_1,\phi_2$ extend on $W^*(\IT^2_\theta)$ as well.

Let $\mathfrak{X}=\{A\in W^*(\IT^2_\theta)|~A=f_{-1}(U)V^{-1}+f_0(U)+f_1(U)V,~f_1,f_0\in L^\infty(\IT), f_{-1}(t):=\overline{f_1(t+\theta)}\}.$
\begin{lmma}\label{closure-WOT}
The subspace $\mathfrak{X}$ is closed in the ultraweak topology. 
\end{lmma}
\begin{proof}
Let $A_\beta:=f^{(\beta)}_{_{-1}}(U)V^{-1}+f^{(\beta)}_{_{0}}(U)+f^{(\beta)}_{_{1}}(U)V$ be a convergent net in the ultraweak topology. Now $\phi_1(A_\beta)=f_{_{0}}^{(\beta)}(U),~\phi_1(A_\beta V)=f^{(\beta)}_{_{-1}}(U)$ and $\phi_1(A_\beta V^{-1})=f^{(\beta)}_{_{1}}(U)$ Since $\phi_1$ is a normal map, which implies that $f_{_{0}}^{(\beta)}(U),~f^{(\beta)}_{_{1}}(U)~and~f^{(\beta)}_{_{-1}}(U)$ (all of which are elements of $L^\infty(\IT)$) are ultraweakly convergent, to $f_0(U),f_1(U),f_{-1}(U)$ (say), and clearly $f_{-1}(t)=\overline{f_1(t+\theta)}.$ 
\end{proof}
\begin{lmma}\label{CHI}
Suppose $f_1(t)f_1(t+\theta)=0$ and $A\in\mathfrak{X}.$ Define $$A_{s,t}:=f_{-1}(e^{2\pi is}U)V^{-1}e^{-2\pi it}+f_0(e^{2\pi is}U)+f_1(e^{2\pi is}U)Ve^{2\pi it}.$$ Suppose $s,s^\prime\in[0,1)$ be such that $|s-s^\prime|\leq\frac{\epsilon}{4}$ where $0<\epsilon<\theta,$ and $|supp(f_1)|<\epsilon,$ where $|C|$ denotes the Lebesgue measure of a Borel subset $C\subseteq\IR.$ Then $A_{s,t}\cdot A_{s^\prime,t^\prime}\in\mathfrak{X}.$
\end{lmma}
\begin{proof}
It suffices to show that the coefficient of $V^2$ in $A_{s,t}\cdot A_{s^\prime,t^\prime}$ is zero. By a direct computation, the coefficient of $V^2$ is $g(l):=f_1(s+l)f_1(s^\prime+l-\theta)e^{2\pi i(t+t^\prime)}.$  But $|(s+l)-(s^\prime+l-\theta)|=|\theta+s-s^\prime|>\epsilon.$ Now by hypothesis, we have $|supp(f_1)|<\epsilon,$ so that $f_1(s+l)\cdot f_1(s^\prime+l-\theta)=0$ and hence the lemma is proved.
\end{proof}
\begin{lmma}\label{V^2=0}
Suppose $A=f_{-1}(U)V^{-1}+f_0(U)+f_1(U)V$ and $f_1(l)f_1(l+\theta)=0,$ for $l\in[0,1).$ Then $A^{2n}\in\mathfrak{X},$ for $n\in\IN.$
\end{lmma}
\begin{proof}
The coefficient of $V^2$ in $A^2$ is $f_1(l)f_1(l+\theta)$ for $l\in[0,1)$ and this is zero by the hypoethesis. Hence $A^2\in\mathfrak{X}.$ The coefficient of $V$ in $A^2$ is $f_1^{(2)}(l):=f_1\left(f_0+\tau_{_\theta}(f_{0})\right),$ where $\tau_{_{\theta}}$ is left translation by $\theta.$ We have $f_1^{(2)}(l)f_1^{(2)}(l+\theta)=0,$ so that applying the same argument as before, we conclude that $A^4\in\mathfrak{X}.$ Proceeding like this we get the required result.
\end{proof}

\begin{lmma}\label{projection in X}
Suppose $P=f_{-1}(U)V^{-1}+f_0(U)+f_1(U)V,$ such that $P^2=P$ and $|supp(f_1)|<\epsilon.$ Then $\left(A_{_{s,t}}(P)\right)\bigwedge\left(A_{_{s^\prime,t^\prime}}(P)\right)\in\mathfrak{X}$ for $|s-s^\prime|<\frac{\epsilon}{4}.$
\end{lmma}
\begin{proof}
We start with the following well-known formula due to von-Neumann:

$$P\wedge Q=SOT-\lim_{n\rightarrow\infty}(P\cdot Q)^n,$$ where $P,Q$ are projections and $P\bigwedge Q$ denotes the projection onto $R(P)\cap R(Q).$

Thus in particular:
$$A_{_{s,t}}(P)\bigwedge A_{_{s^\prime,t^\prime}}(P)=SOT-\lim_{n\rightarrow\infty}\{A_{_{s,t}}(P)\cdot A_{_{s^\prime,t^\prime}}(P)\}^n.$$ Now by the hypothesis, $|s-s^\prime|<\frac{\epsilon}{4}$ and $|supp(f_1)|<\epsilon.$ It follows by Lemma \ref{CHI}, that\\
$A_{_{s,t}}(P)\cdot A_{_{s^\prime,t^\prime}}(P)\in\mathfrak{X}.$ The coefficient of $V$ in 
$A_{_{s,t}}(P)\cdot A_{_{s^\prime,t^\prime}}(P)$ is 
$$f_1^{(2)}(l):=\{f_1(s+l)f_0(s^\prime+t-\theta)e^{2\pi it}+f_0(s+l)f_1(s^\prime+t)e^{2\pi it^\prime}\}.$$ One may check that $f_1^{(2)}(l)f_1^{(2)}(l+\theta)=0$ for $|s-s^\prime|<\frac{\epsilon}{4}.$ Thus by Lemma \ref{V^2=0},\\ $\{A_{_{s,t}}(P)\cdot A_{_{s^\prime,t^\prime}}(P)\}^{2n}\in\mathfrak{X}$ for $n\geq1.$ Now by Lemma \ref{closure-WOT}, the subspace $\mathfrak{X}$ is closed in the SOT topology. Thus $$SOT-\lim_{n\rightarrow\infty}\{A_{_{s,t}}(P)\cdot A_{_{s^\prime,t^\prime}}(P)\}^{2n}\in\mathfrak{X};$$ i.e. 
$\left(A_{_{s,t}}(P)\right)\bigwedge\left(A_{_{s^\prime,t^\prime}}(P)\right)\in\mathfrak{X}.$
\end{proof}
\begin{lmma}\label{solve for projection}
Let $P=f_{-1}(U)V^{-1}+f_0(U)+f_1(U)V$ and $A=f_{_{-1}}^{(A)}(U)V^{-1}+f_{_0}^{(A)}(U)+f_{_1}^{(A)}(U)V$ be projections, $(f_{-1},f_0,f_1)$ and $(f_{_{-1}}^{(A)},f_{_0}^{(A)},f_{_{1}}^{(A)})$ satisfying the conditions given in [\cite{davidson},page 173]. Then $A\leq A_{_{s,t}}(P)$ and $A\leq A_{_{s^\prime,t^\prime}}(P)$ if and only if the following hold:
\begin{itemize}
\item $f_1(s+l)f_{_1}^{(A)}(l-\theta)=0;$
\item $f_{-1}(s+l)f_{_{-1}}^{(A)}(l+\theta)=0;$
\item $f_0(s+l)f_{_0}^{(A)}(l)+f_1(s+l)f_{_{-1}}^{(A)}(l-\theta)e^{2\pi it}+f_{-1}(s+l)f_{_{1}}^{(A)}(l+\theta)e^{-2\pi it}=f_{_0}^{(A)}(l);$
\item $f_1(s+l)f_{_{0}}^{(A)}(l-\theta)e^{2\pi it}+f_0(s+l)f_{_{1}}^{(A)}(l)=f_{_{1}}^{(A)}(l);$
\item $f_{-1}(s+l)f_{_{0}}^{(A)}(l+\theta)e^{-2\pi i t}+f_0(s+l)f_{_{-1}}^{(A)}(l)=f_{_{-1}}^{(A)}(l);$
\item $f_1(s^\prime+l)f_{_1}^{(A)}(l-\theta)=0;$
\item $f_{-1}(s^\prime+l)f_{_{-1}}^{(A)}(l+\theta)=0;$
\item $f_0(s^\prime+l)f_{_0}^{(A)}(l)+f_1(s^\prime+l)f_{_{-1}}^{(A)}(l-\theta)e^{2\pi it^\prime}+f_{-1}(s^\prime+l)f_{_{1}}^{(A)}(l+\theta)e^{-2\pi it^\prime}=f_{_0}^{(A)}(l);$
\item $f_1(s^\prime+l)f_{_{0}}^{(A)}(l-\theta)e^{2\pi it^\prime}+f_0(s^\prime+l)f_{_{1}}^{(A)}(l)=f_{_{1}}^{(A)}(l);$
\item $f_{-1}(s^\prime+l)f_{_{0}}^{(A)}(l+\theta)e^{-2\pi i t^\prime}+f_0(s^\prime+l)f_{_{-1}}^{(A)}(l)=f_{_{-1}}^{(A)}(l);$
\end{itemize}
for $l\in[0,1).$
\end{lmma}
\begin{proof}
It follows by comparing the coefficients of $V^{-1},V$ and $1$ from the equations
\begin{equation*}
A_{_{s,t}}(P)A=A;~~A_{_{s^\prime,t^\prime}}(P)A=A.
\end{equation*}
\end{proof}
\begin{lmma}\label{B<A}
For two projections $A$ and $B$ such that $$A=f^{(A)}_{_{-1}}(U)V^{-1}+f^{(A)}_{_{0}}(U)+f^{(A)}_{_{1}}(U)V,$$
$$B=f^{(B)}_{_{-1}}(U)V^{-1}+f^{(B)}_{_{0}}(U)+f^{(B)}_{_{1}}(U)V;$$ we have $A\leq B$ if and only if 
\begin{itemize}
\item $\fonedown{(B)}(l)\fonedown{(A)}(l-\theta)=0;$
\item $\fonedown{(B)}(l+\theta)\fonedown{(A)}(l+2\theta)=0;$
\item $\fzerodown{(B)}(l)\fzerodown{A}(l)+\fonedown{(B)}(l)\fzerodown{(A)}(l)+\fonedown{(B)}(l+\theta)\fonedown{(A)}(l+\theta)=\fzerodown{(A)}(l);$
\item $\fonedown{(B)}(l)\fzerodown{(A)}(l-\theta)+\fzerodown{(B)}(l)\fonedown{(A)}(l)=\fonedown{(A)}(l);$
\item $\fonedown{(B)}(l+\theta)\fzerodown{(A)}(l+\theta)+\fzerodown{(B)}(l)\fonedown{(A)}(l+\theta)=\fonedown{(A)}(l+\theta);$
\end{itemize}
for $l\in[0,1).$
\end{lmma}
\begin{proof}
It follows by comparing the coefficients of $V,V^{-1},1$ in the equation $BA=A.$
\end{proof}
\begin{lmma}\label{000}
Let $P=f_{-1}(U)V^{-1}+f_0(U)+f_1(U)V$ such that $P$ is a projection and suppose $f_0(t)=0$ for some $t.$ Then $f_1(t)=f_1(t+\theta)=0.$
\end{lmma}
\begin{proof}
The fact that $P^2=P$ implies that 
\begin{equation}\label{i10}
\begin{split}
f_0(t)-\left(f_0(t)\right)^2&=|f_1(t-\theta)|^2+|f_1(t)|^2~\mbox{(see \cite{davidson},page 173)},\\
f_0(t+\theta)-\left(f_0(t+\theta)\right)^2&=|f_1(t)|^2+|f_1(t+\theta)|^2.
\end{split}
\end{equation}
The first expression in (\ref{i10}) implies that $f_1(t)=0.$ Moreover we have 
\begin{equation*}
f_1(t+\theta)\left(1-f_0(t)-f_0(t+\theta)\right)=0~\cite[page 173]{davidson};
\end{equation*}
so that if $f_0(t+\theta)=0$ implies $f_1(t+\theta)=0;$ else if $f_0(t+\theta)=1,$ the second expression in (\ref{i10}) gives $f_1(t+\theta)=0.$
\end{proof}
For a set $A\subseteq\IR$ and real numbers $a\in\IR,$ $\tau_a(A):=A+a.$

Define functions $f_{_0}$ and $f_{_1}$ by:

$f_{_{0}}(t)=
\begin{cases}
\epsilon^{-1}t~~if~ 0\leq t\leq\epsilon\\
1~~if~\epsilon\leq t\leq\theta\\
\epsilon^{-1}(\theta+\epsilon-t)~if~\theta\leq t\leq\theta+\epsilon\\
0~if~\theta+\epsilon\leq t\leq1 
\end{cases}
$

$f_{_{1}}(t)=
\begin{cases}
\sqrt{f_{_{0}}(t)-f_{_{0}}(t)^2}~if~\theta\leq t\leq\theta+\epsilon\\
0~if~otherwise.
\end{cases}
$

It is known (see \cite{davidson}) that $P:=f_{_{-1}}(U)V^{-1}+f_{_0}(U)+f_{_1}(U)V$ is a projection in $C^*(\IT^2_\theta).$

\bthm\label{stoptime}
Let $P=f_{_{-1}}(U)V^{-1}+f_{_0}(U)+f_{_{1}}(U)V$ be a projection with $f_{_0},f_{_1}$ as described above. Consider the projections 
$A_{_{s,t}}(P),~A_{_{s^\prime,t^\prime}}(P)$ such that $|s-s^\prime|<\frac{\epsilon}{4}.$ Then 
\begin{equation*}
\left(A_{_{s,t}}(P)\right)\bigwedge\left(A_{_{s^\prime,t^\prime}}(P)\right)=
\chi_{_{S}}(U),
\end{equation*}
for the set $S=X_1\cap X_2\cap X_3\cap X_4,$ where $X_1=\tau_{-s}(\{x|f_1(x)=0\}),X_2:=\tau_{-s^\prime}(\{x|f_1(x)=0\}),$ $X_3:=\tau_{-s}(\{x|f_0(x)=1\})$ and $X_4:=\tau_{-s^\prime}(\{x|f_0(x)=1\}).$ 
\ethm
\begin{proof}
The hypothesis of the theorem and Lemma \ref{projection in X} together implies that 
$$\left(A_{_{s,t}}(P)\right)\bigwedge\left(A_{_{s^\prime,t^\prime}}(P)\right)\in\mathfrak{X}.$$ 
Let $B=\chi_{_{S}}(U).$ Then it follows that the conditions of Lemma \ref{solve for projection} hold with $f_1^{(A)}=0$ and\\ $f_0^{(A)}(U)=\chi_{_{S}}(U).$ Thus $B\leq A_{_{s,t}}(P),$ $B\leq A_{_{s^\prime,t^\prime}}(P).$ Again if $A=f_{-1}^{(A)}(U)V^{-1}+f_0^{(A)}(U)+f_1^{(A)}(U)V$ be a projection, then it may be easily observed that $A\leq A_{_{s,t}}(P)$ and $A\leq A_{_{s^\prime,t^\prime}}(P)$ together with Lemma \ref{000} implies that $f_1,f_0$ is zero outside $S.$ An application of Lemma \ref{B<A} implies that $f_1,f_0$ must vanish outside $S$ if and only if $A\leq B.$ Hence the theorem is proved.

\end{proof}
It is worthwhile to note that the conclusion of the above theorem holds if we replace $U$ by $U^k,$ $V$ by $V^k,$ and $\theta$ by $\{k\theta\}$ ($\{\cdot\}$ denoting the fractional part).
\paragraph{}
Let $P_n=f^{(k_n)}_{_{-1}}(U^{k_n})+\fzerodown{(k_n)}(U^{k_n})+\fonedown{(k_n)}(U^{k_n})U^{k_n},$ be projections as described in \cite[page 173]{davidson} such that $\{k_n\theta\}\rightarrow0.$ Put $\epsilon_n:=\frac{\{k_n\theta\}}{2}.$ Consider a standard Brownian motion in $\IR^2,$ given by $(W_t^{(1)},W_t^{(2)}).$ Define $j_t:W^*(\IT^2_\theta)\rightarrow W^*(\IT^2_\theta)\ot B(\Gamma(L^2(\IR_+,\IC^2)))$ by $j_t(\cdot):=\alpha_{_{(e^{2\pi iW_t^{(1)}},e^{2\pi iW_t^{(2)})}}}(\cdot).$
\bthm
Almost surely, $\bigwedge_{s\leq t}\left(j_s(P_n)(\omega)\right)\in W^*(U),$ for all $n,$ i.e. 
$$\bigwedge_{s\leq t}\left(j_s(P_n)\right)\in W^*(U)\ot B(\Gamma(L^2(\IR_+,\IC^2))),$$ for each $n.$
\ethm
\begin{proof}
In the strong operator topology, 
\begin{equation}\label{uniform limit}
\bigwedge_{0\leq s\leq t}\left(j_s(P_n)\right)=\lim_{m\rightarrow\infty}\bigwedge_{i}\{j_{\frac{it}{2^m}}(P_n)\wedge j_{\frac{(i+1)t}{2^m}}(P_n)\}.
\end{equation}
Now almost surely a Brownian path restricted to $[0,t]$ is uniformly continuous, so that the for sufficiently large $m,$ and for almost all $\omega,$ $|W^{(1)}_{\frac{it}{2^m}}-W^{(1)}_{\frac{(i+1)t}{2^m}}|$ can be made small, uniformly for all $i$ such that $i=0,1,..2^m.$ So 
$\bigwedge_{i}\{j_{\frac{it}{2^m}}(P_n)\wedge j_{\frac{(i+1)t}{2^m}}(P_n)\}\in W^*(U)\cap\mathfrak{X}$ by Theorem \ref{stoptime}. Now Lemma \ref{closure-WOT} implies that $W^*(U)\cap\mathfrak{X}$ is closed in the WOT-topology. Thus $$\lim_{m\rightarrow\infty}\bigwedge_{i}\{j_{\frac{it}{2^m}}(P_n)\wedge j_{\frac{(i+1)t}{2^m}}(P_n)\}\in W^*(U)\cap\mathfrak{X}.$$
\end{proof}
Let $z_n=e^{^{2\pi i\frac{3\{k_n\theta\}}{4}}}.$ Consider the sequence of states
$\phi_{z_n}:=ev_{z_n}\circ E_1.$ By \cite{lingaraj-dg}, this is a sequence of pure states on $C^*(\IT^2_\theta)$ converging in the weak-$^\ast$ topology to $\phi_1:=ev_{1}\circ E_1.$  Following the discussion in the beginning, consider
$$\innerl e(0),(\phi_{z_n}\ot1)\circ\bigwedge_{0\leq s\leq t}(j_s(P_n))e(0)\innerr.$$ A direct computation shows that this is equal to $$\IP\{e^{2\pi iW^{(1)}_s}\in\clb,~0\leq s\leq t\}=\IP\{\tau_{_{[\frac{-\{k_n\theta\}}{4},\frac{\{k_n\theta\}}{4}]}}>t\},$$ where $\clb:=\{e^{2\pi i x}:~x\in[\frac{-\{k_n\theta\}}{4},\frac{\{k_n\theta\}}{4}]\}.$ So we have a family of $(\tau_n)_n$ random times defined by
$$\tau_n([t,+\infty))=\bigwedge_{0\leq s\leq t}(j_s(P_n));$$ so that $\int_0^t\innerl e(0),(\phi_{z_n}\ot1)\circ\bigwedge_{0\leq s\leq t}(j_s(P_n))e(0)\innerr dt$ can be taken as the expectation of the random time $\tau_n.$ Note that here the analogue for balls of decreasing volume is $(P_n)_n,$ such that\\ $tr(P_n)=\{k_n\theta\}\rightarrow0,$ $tr$ being the canonical trace in $W^*(\IT^2_\theta).$  Now, by Proposition \ref{pinsky}, we have
\begin{equation}
\begin{split}
&\int_0^t\innerl e(0),(\phi_{z_n}\ot1)\circ\bigwedge_{0\leq s\leq t}(j_s(P_n))e(0)\innerr dt\\
&=\IE(\tau_{_{[\frac{-\{k_n\theta\}}{4},\frac{\{k_n\theta\}}{4}]}})\\
&=2\sin^2\left(\frac{\{k_n\theta\}}{8}\right)+
\frac{2}{3}\sin^4\left(\frac{\{k_n\theta\}}{8}\right)+O\left(\sin^5\left(\frac{\{k_n\theta\}}{8}\right)\right)\\
&=\frac{\{k_n\theta\}^2}{2^5}+\frac{\{k_n\theta\}^4}{2^{11}.3}+O(\{k_n\theta\}^5),~\mbox{since the mean curvature of the circle viewed inside $\IR^2$ is $1.$}
\end{split}
\end{equation}
\brmrk
In view of equations (\ref{intrinsic dimension}),(\ref{extrinsic dimension}) and (\ref{formula for mean curvature1}), we see that the `intrinsic dimension' $n_0=1,$ the `extrinsic diimension' $d=5,$ and the `mean curvature' is $\frac{1}{2\sqrt{2}}.$ As we have already remarked in the introduction, the instrinsic one-dimensionality may be interpreted as a manifestation of the local one-dimensionality of the `leaf space' of the Kronecker foliation (see \cite{connes} for details). It is worth pointing out that the spectral behaviour  of the standard Dirac operator or the Laplacian coming from it for this noncommutative manifold is identical with that of the commutative two-torus, and thus it does not recognize the  one-dimensionality of the leaf space of Kronecker foliation. Thus, it is a remarkable success of our (quantum) stochastic analysis using exit time to reveal the association of the noncommutative geometry of $\cla_\theta$ with the leaf space of Kronecker foliation, and also to distinguish it from the commutative two-torus. All these give  a good justification for developing a general theory of   quantum stochastic geometry.
\ermrk

\end{document}